\documentclass[11pt,a4paper]{article}
\usepackage{authblk}
\usepackage{bm}
\usepackage{geometry}
\usepackage{fancyhdr}
\usepackage{amsmath,amsthm,amssymb,amsfonts,mathrsfs,amscd}

\usepackage{enumerate}

\newtheorem{thm}{Theorem}[section]
\newtheorem{prop}[thm]{Proposition}
\newtheorem{lem}[thm]{Lemma}
\newtheorem{cor}[thm]{Corollary}
\usepackage{cite}

\theoremstyle{definition}
\newtheorem*{abs}{Abstract}
\newtheorem{definition}[thm]{Definition}
\newtheorem{example}[thm]{Example}

\newtheorem{remark}[thm]{Remark}

\numberwithin{equation}{section}

\begin{document}
	
\title{Power Stability of Finitely Generated Dynamical Frames for Positive Operators}
	\author{Jian Wu\thanks{Corresponding author:
xingxingwu2022@163.com}}

	\affil{College of Mathematics, Southwestern University of Finance and Economics, Chengdu 611130, P.R. China}

	\renewcommand*{\Affilfont}{\small\it}
	\renewcommand\Authands{ and }
	\date{}

	\maketitle

	\begin{abs}
Let
\(
\Lambda=\{\lambda_j\}_{j\in\mathbb N}\subset[0,1)
\)
be a finite union of Carleson sequences, and let
\(
\Lambda^p=\{\lambda_j^p\}_{j\in\mathbb N}
\)
for \(p>0\). We prove that \(\Lambda\) and \(\Lambda^p\) admit
compatible Hartmann-type cluster decompositions and that, under the
natural coordinate identification,
\(
H^2(\Lambda)=H^2(\Lambda^p).
\)
After taking account of the change in the normalization weights, this
trace-space identity gives
\(
\operatorname{Ran}(T_\Lambda)
=
\operatorname{Ran}(T_{\Lambda^p})
\)
for the associated weighted evaluation operators. The characterization of finitely generated dynamical frames for normal
operators then implies that, for every positive operator
\(A\in\mathcal B(H)\), every finite index set \(I\), and every \(p>0\),
\(
\{A^nf_i\}_{n\in\mathbb N,\,i\in I}
\text{ is a frame for }H\) if and only if
\(
\{A^{pn}f_i\}_{n\in\mathbb N,\,i\in I}
\text{ is a frame for }H.
\)
In particular, for every \(q\in\mathbb N^+\), the subfamily
\(
\{A^{qn}f_i\}_{n\in\mathbb N,\,i\in I}
\)
is still a frame, and hence every such dynamical frame has infinite
excess.
An example shows that the positivity assumption cannot in general be
replaced by normality.

		\vskip0.05in
		
		\vskip0.05in
		
	\end{abs}
	
\section{Introduction}
Dynamical sampling studies the recovery of a signal from measurements
taken at different spatial locations and times while the signal evolves
under an operator. If \(E\) is a bounded evolution operator on a Hilbert
space \(H\), \(x\in H\) is an initial state, and
\(\{f_i\}_{i\in I}\subset H\) is a family of sampling vectors, then
\[
\langle E^n x,f_i\rangle
=
\langle x,(E^*)^nf_i\rangle,
\qquad n\in\mathbb N,
\quad i\in I.
\]
Thus, stable recovery is equivalent to the frame property of an
operator-orbit system. This point of view was developed in
\cite{ADKTimeSpace,ADKExact,AAldroubiDynamicalsampling}; see also
\cite{DynamicalSamplingSurvey} and the broader sampling literature
\cite{AAldroubiKGrochenig,KGrochenig,AAldroubiABaskakov,BAdcock,FBass}.
We fix some notation used throughout the paper. Let
\(
\mathbb N:=\{0,1,2,\ldots\},
\
\mathbb D:=\{z\in\mathbb C:|z|<1\}.
\)
The Hardy space \(H^2(\mathbb D)\) consists of all analytic functions
\(
f(z)=\sum_{n=0}^{\infty}c_nz^n
\)
such that
\(
\|f\|_{H^2(\mathbb D)}^2
:=
\sum_{n=0}^{\infty}|c_n|^2
<\infty.
\)
We also write
\(
H^\infty(\mathbb D)
:=
\left\{
f\in\operatorname{Hol}(\mathbb D):
\|f\|_\infty
:=
\sup_{z\in\mathbb D}|f(z)|
<\infty
\right\}.
\)

If
\(
\Gamma=\{\gamma_j\}_{j\in\mathbb N}\subset\mathbb D
\)
is an indexed sequence, we denote the corresponding trace sequence
spaces by
\(
H^2(\Gamma)
:=
\left\{
\{f(\gamma_j)\}_{j\in\mathbb N}:
f\in H^2(\mathbb D)
\right\}
\)
and
\(
H^\infty(\Gamma)
:=
\left\{
\{f(\gamma_j)\}_{j\in\mathbb N}:
f\in H^\infty(\mathbb D)
\right\}.
\)
Whenever two indexed sequences are compared, equality of their trace
spaces is understood under the coordinate identification induced by
their enumerations.
The theory of frames generated by operator iterations has subsequently
been developed for several classes of operators. Aldroubi, Cabrelli,
Molter, and Tang obtained a fundamental characterization for diagonal
self-adjoint operators and showed the decisive role of Carleson
interpolation \cite{AAldroubiDynamicalsampling}. Aldroubi, Cabrelli, \c{C}akmak, Molter, and
Petrosyan studied iterative actions of general normal operators and
analyzed completeness, Bessel sequences, bases, and frames through
spectral theory with multiplicity \cite{NormalOperators}. Further
results on Bessel orbits, operator representations, and the structural
properties of frames generated by operator orbits can be found in
\cite{DynamicalSamplingSystems,BesselOrbits,OperatorRepresentations,
BoundedRepresentations,IteratedSystems,OperatorOrbits}.
Of particular relevance to the present paper is the finite-generator
system
\(
\{A^nf_i\}_{n\in\mathbb N,\,i\in I},
\ |I|<\infty.
\)
Cabrelli, Molter, Paternostro, and Philipp characterized when this
system is a frame for a normal operator \(A\) \cite{CCabrelliUMolter}. Their characterization combines spectral multiplicity, the requirement
that the spectral set be a finite union of Carleson sequences, uniform
estimates on the spectral subspaces, and a range condition for a
weighted evaluation operator on \(H^2(\mathbb D)\). Related descriptions using the pseudo-hyperbolic
geometry of \(\mathbb D\), interpolating Blaschke products, and model
spaces were obtained in
\cite{MultiOrbitalFrames,VectorValuedModelSpaces}. The underlying
function theory goes back to
\cite{Carleson1958,ShapiroShields,Duren1970,Hoffman1988,AHartmann,Hartmann,DurenSchuster}.

Continuous powers of normal operators were studied in
\cite{ContinuousPowers}. For single-generator Carleson frames,
Christensen et al.~\cite{OChristensenMHasannasab} established stability
under integer subsampling. Krishtal and Miller
\cite{KrishtalMiller} considered systems of the form
\(
\{T^{Nk+j_k}\varphi\}_{k\in\mathbb N},
\
0\leq j_k<N,
\)
where the offsets \(j_k\) may be nonintegral. Their result allows
fractional exponents but concerns a single generator and retains an
integer coarse step \(N\). Here we consider finitely many generators
and the uniform power transformation \(A\mapsto A^p\) for arbitrary
\(p>0\).

The passage from one generator to finitely many is not a routine
generator-by-generator extension. In fact, for a finitely generated
dynamical frame, the individual orbits need not separately satisfy the
frame condition. Moreover, the spectral set occurring in
\cite[Theorem~2.3]{CCabrelliUMolter} is only required to be a finite
union of Carleson sequences, and the characterization contains a range
condition that couples the generators through the spectral subspaces.
Consequently, the known single-orbit results do not directly address
the multi-generator problem considered here.

Our main analytic result is the invariance of Hardy-space traces under
positive power transformations. More precisely, if
\(
\Lambda=\{\lambda_j\}_{j\in\mathbb N}\subset[0,1)
\)
is a finite union of Carleson sequences and
\(
\Lambda^p=\{\lambda_j^p\}_{j\in\mathbb N},
\ p>0,
\)
then, under the natural coordinate identification,
\(
H^2(\Lambda)=H^2(\Lambda^p).
\)
The proof combines a common Hartmann cluster decomposition for
\(\Lambda\) and \(\Lambda^p\) with uniform comparisons between ordinary
Newton divided differences and their pseudo-hyperbolic counterparts.
The resulting trace-space identity is stronger than the fact that the
power map preserves finite unions of Carleson sequences and is
independent of the subsequent frame application.

The power transformation also changes the normalization weights from
\((1-\lambda_j^2)^{1/2}\) to
\((1-\lambda_j^{2p})^{1/2}\). We prove that the ratio of these weights
and its reciprocal belong to \(H^\infty(\Lambda)\). Together with the
trace-space identity, this yields
\(
\operatorname{Ran}(T_\Lambda)
=
\operatorname{Ran}(T_{\Lambda^p}),
\)
as well as the corresponding equality for finite direct sums. This
range identity is the key additional ingredient required by the
coupled condition in the finite multi-orbit characterization.

Applying the characterization of Cabrelli, Molter, Paternostro, and
Philipp, we obtain, for every positive operator \(A\), every finite
index set \(I\), and every \(p>0\),
\[
\{A^nf_i\}_{n\in\mathbb N,\,i\in I}
\text{ is a frame for }H
\quad\Longleftrightarrow\quad
\{A^{pn}f_i\}_{n\in\mathbb N,\,i\in I}
\text{ is a frame for }H.
\]
This gives integer-time subsampling and infinite excess for every such
finitely generated dynamical frame. A final example shows that positivity
cannot in general be replaced by normality.

We review the concept of frames. A sequence
\(\{f_\gamma\}_{\gamma\in G}\) in a Hilbert space \(H\) is a frame for
\(H\) if there exist constants \(0<a\leq b<\infty\) such that
\(
a\|x\|^2
\leq
\sum_{\gamma\in G}
|\langle x,f_\gamma\rangle|^2
\leq
b\|x\|^2,
\ x\in H.
\)
If only the upper inequality holds, the sequence is called a Bessel
sequence.
Throughout, \(H\) denotes an infinite-dimensional separable Hilbert
space, and we set
\(
\mathbb N^+:=\mathbb N\setminus\{0\}.
\)
The unit circle is denoted by
\(
\mathbb T:=\{z\in\mathbb C:|z|=1\}.
\)

   For \(\lambda\in\mathbb D\), define the Blaschke factor
\[
b_\lambda(z)
=
\begin{cases}
\dfrac{|\lambda|}{\lambda}
\dfrac{\lambda-z}{1-\overline{\lambda}z},
& \lambda\ne0,\\[2ex]
z,
& \lambda=0.
\end{cases}
\]
Thus
\(
|b_\lambda(z)|
=
\rho(\lambda,z)
:=
\left|
\frac{\lambda-z}{1-\overline{\lambda}z}
\right|,
\)
where \(\rho\) denotes the pseudo-hyperbolic distance.

Let
\(
\Lambda=\{\lambda_j\}_{j\in\mathbb N}\subset\mathbb D
\)
be a sequence of distinct points. All sequences considered below are
infinite unless explicitly stated otherwise. We say that \(\Lambda\)
satisfies the Carleson condition if
\(
\delta_\Lambda
:=
\inf_{j\in\mathbb N}
\prod_{\substack{i\in\mathbb N\\ i\ne j}}
|b_{\lambda_i}(\lambda_j)|
>0.
\)
In this case, we write \(\Lambda\in\mathcal C\).

The remainder of the paper is organized as follows. Section~2 studies
Carleson sequences under power transformations and constructs a common
Hartmann-type cluster decomposition. Section~3 uses Newton interpolation
to establish the invariance of the associated \(H^2\)-trace spaces.
Section~4 proves the corresponding range identity for weighted
evaluation operators and applies it to power stability and redundancy
properties of finitely generated dynamical frames. A counterexample
shows that positivity cannot in general be omitted.

\section{Power Perturbations and Common Cluster Decompositions}
The main result of this section is that, for every \(p>0\), a finite
union of Carleson sequences
\(\Lambda\subset(0,1)\) and its power transform \(\Lambda^p\)
admit compatible Hartmann-type cluster decompositions. As a preliminary
step, we prove that the Carleson condition is preserved under arbitrary
positive powers.

The following decomposition result is
Proposition~1.1 of Hartmann~\cite{Hartmann}.

\begin{lem}
\label{cc}
Let
\(
\Lambda=\bigcup_{r=1}^{N}\Lambda_r\) for \(\Lambda_r\in\mathcal C\) and \(1\le r\le N\).
Then, for every \(0<\eta<1\), there exists a partition
\(
\Lambda
=
\mathop{\dot\bigcup}_{n\in\mathbb N^+}\sigma_n
\)
with the following properties:
\begin{enumerate}[\rm (i)]
\item
Each \(\sigma_n\) is finite and
\(
\sup_{n\in\mathbb N^+}|\sigma_n|\le N,
\)
where \(|E|\) denotes the cardinality of a finite set \(E\).

\item
For every \(n\in\mathbb N^+\) and all distinct
\(\lambda,\mu\in\sigma_n\),
\(
\rho(\lambda,\mu)
=
|b_\lambda(\mu)|
<
\eta.
\)

\item
There exists a constant \(\delta>0\) such that, for every
choice of points
\(
\lambda_n^0\in\sigma_n
\ (n\in\mathbb N^+),
\)
the selector
\(
\Lambda_0
=
\{\lambda_n^0:n\in\mathbb N^+\}
\)
belongs to \(\mathcal C\) and satisfies
\(
\delta_{\Lambda_0}\ge\delta.
\)

\item
There exist a constant \(M_D<\infty\) and functions
\(
D_n\in H^\infty(\mathbb D),
\ n\in\mathbb N^+,
\)
such that
\[
D_n(\lambda)
=
\begin{cases}
1, & \lambda\in\sigma_n,\\
0, & \lambda\in\Lambda\setminus\sigma_n,
\end{cases}
\]
and
\(
\sum_{n\in\mathbb N^+}|D_n(z)|
\le M_D\) for \( z\in\mathbb D\).
\end{enumerate}
\end{lem}
For a sequence
\(
\Lambda=\{\lambda_j\}_{j\in\mathbb N}\subset(0,1)
\)
and \(p>0\), set
\(
\Lambda^p:=\{\lambda_j^p:j\in\mathbb N\}.
\)

\begin{lem}
\label{aa}
Let \(\Lambda\subset(0,1)\) and assume that
\(
a:=\inf_{\lambda\in\Lambda}\lambda>0.
\)
Then, for every \(p>0\), there exist constants
\(
0<C_1\le C_2<\infty
\)
such that, for all distinct \(\lambda,\mu\in\Lambda\),
\(
C_1
\le
\frac{\rho(\lambda^p,\mu^p)}
{\rho(\lambda,\mu)}
\le
C_2.
\)
\end{lem}

\begin{proof}
Let \(\lambda,\mu\in\Lambda\) be distinct. Since
\(
\rho(\lambda,\mu)
=
\frac{|\lambda-\mu|}{1-\lambda\mu},
\)
we have
\(
\frac{\rho(\lambda^p,\mu^p)}
{\rho(\lambda,\mu)}
=
\frac{|\lambda^p-\mu^p|}
{|\lambda-\mu|}
\frac{1-\lambda\mu}
{1-\lambda^p\mu^p}.
\)
By the mean value theorem, there exists a point \(\xi\) between
\(\lambda\) and \(\mu\) such that
\(
\frac{|\lambda^p-\mu^p|}
{|\lambda-\mu|}
=
p\xi^{p-1}.
\)
Since \(a\le\xi<1\), it follows that
\(
p\min\{1,a^{p-1}\}
\le
\frac{|\lambda^p-\mu^p|}
{|\lambda-\mu|}
\le
p\max\{1,a^{p-1}\}.
\)
Moreover, for every \(0<t<1\),
\(
\min\{1,p\}
\le
\frac{1-t^p}{1-t}
\le
\max\{1,p\}.
\)
Applying this inequality to \(t=\lambda\mu\), we obtain
\(
\frac{1}{\max\{1,p\}}
\le
\frac{1-\lambda\mu}
{1-\lambda^p\mu^p}
\le
\frac{1}{\min\{1,p\}}.
\)
Combining the preceding estimates gives
\[
\frac{p\min\{1,a^{p-1}\}}
{\max\{1,p\}}
\le
\frac{\rho(\lambda^p,\mu^p)}
{\rho(\lambda,\mu)}
\le
\frac{p\max\{1,a^{p-1}\}}
{\min\{1,p\}}.
\]
Thus the assertion holds with
\(
C_1
=
\frac{p\min\{1,a^{p-1}\}}
{\max\{1,p\}},
\
C_2
=
\frac{p\max\{1,a^{p-1}\}}
{\min\{1,p\}}.
\)
\end{proof}

The integer-power case follows from
\cite[Theorem~2.1]{OChristensenMHasannasab}, while the real-power case
also follows indirectly from
\cite[Theorem~2.3 and Corollary~2.4]{KrishtalMiller}, together with
the standard characterization of Carleson frames
\cite[Theorem~3.16]{AAldroubiDynamicalsampling}.
The following proposition gives a direct proof and an explicit lower
bound for \(\delta_{\Lambda^p}\).
\begin{prop}
\label{bb}
Let \(\Lambda\in\mathcal C\) and assume that
\(\Lambda\subset(0,1)\). Then, for every \(p>0\),
\(
\Lambda^p\in\mathcal C.
\)
\end{prop}

\begin{proof}
Since the mapping \(t\mapsto t^p\) is strictly increasing on
\((0,1)\), the points of \(\Lambda^p\) are distinct.

For distinct \(i,j\in\mathbb N\), put
\(
\rho_{ij}
:=
|b_{\lambda_i}(\lambda_j)|
=
\rho(\lambda_i,\lambda_j)
\)
and
\(
\rho_{ij}^{(p)}
:=
|b_{\lambda_i^p}(\lambda_j^p)|
=
\rho(\lambda_i^p,\lambda_j^p).
\)
Let
\(
x_{ij}:=\max\{\lambda_i,\lambda_j\},
y_{ij}:=\min\{\lambda_i,\lambda_j\}.
\)
Then
\(
\rho_{ij}
=
\frac{x_{ij}-y_{ij}}{1-x_{ij}y_{ij}},
\)
and hence
\(
1-\rho_{ij}
=
\frac{(1-x_{ij})(1+y_{ij})}
{1-x_{ij}y_{ij}}.
\)
Similarly,
\(
1-\rho_{ij}^{(p)}
=
\frac{(1-x_{ij}^p)(1+y_{ij}^p)}
{1-x_{ij}^py_{ij}^p}.
\)
Therefore,
\begin{equation}
\label{eq:one-minus-rho-ratio}
\frac{1-\rho_{ij}^{(p)}}{1-\rho_{ij}}
=
\frac{1-x_{ij}^p}{1-x_{ij}}
\cdot
\frac{1+y_{ij}^p}{1+y_{ij}}
\cdot
\frac{1-x_{ij}y_{ij}}
{1-x_{ij}^py_{ij}^p}.
\end{equation}

For every \(0<t<1\), one has
\begin{equation}
\label{eq:power-boundary-comparison}
\min\{1,p\}
\leq
\frac{1-t^p}{1-t}
\leq
\max\{1,p\}.
\end{equation}
Moreover,
\(
\frac{1+y_{ij}^p}{1+y_{ij}}
\leq 2.
\)
Applying \eqref{eq:power-boundary-comparison} to \(x_{ij}\) and
\(x_{ij}y_{ij}\), and using
\eqref{eq:one-minus-rho-ratio}, we obtain
\begin{equation}
\label{eq:one-minus-rho-comparison}
1-\rho_{ij}^{(p)}
\leq
K_p(1-\rho_{ij}),
\end{equation}
where
\(
K_p
:=
2\frac{\max\{1,p\}}{\min\{1,p\}}.
\)

Fix \(j\in\mathbb N\). Since \(\Lambda\in\mathcal C\),
\(
\prod_{i\ne j}\rho_{ij}
\geq
\delta_\Lambda.
\)
Because every factor \(\rho_{ij}\) belongs to \((0,1]\), the product
is no larger than any one of its factors. Thus for \(i\ne j\),
\(
\rho_{ij}\geq\delta_\Lambda
\). Since \(\Lambda\in\mathcal C\), it has no accumulation point in
\(\mathbb D\). As \(\Lambda\subset(0,1)\), it follows that
\(
\inf_{\lambda\in\Lambda}\lambda>0.
\)
Therefore Lemma~\ref{aa} applies to \(\Lambda\).
By Lemma~\ref{aa}, there exists a constant \(C_1>0\), independent of
\(i\) and \(j\), such that
\(
\rho_{ij}^{(p)}
\geq
C_1\rho_{ij}
\geq
C_1\delta_\Lambda.
\)
Set
\(
\eta_0
:=
\min\left\{\frac12,C_1\delta_\Lambda\right\}>0.
\)
Then
\begin{equation}
\label{eq:rho-p-uniform-lower-bound}
\rho_{ij}^{(p)}\geq\eta_0,
\qquad i\ne j.
\end{equation}

For every \(t\in[\eta_0,1]\),
\(
-\log t
=
\int_t^1\frac{ds}{s}
\leq
\frac{1-t}{\eta_0}.
\)
It follows from \eqref{eq:rho-p-uniform-lower-bound} and
\eqref{eq:one-minus-rho-comparison} that
\[
\sum_{i\ne j}-\log\rho_{ij}^{(p)}
\leq
\frac{1}{\eta_0}
\sum_{i\ne j}\left(1-\rho_{ij}^{(p)}\right)
\leq
\frac{K_p}{\eta_0}
\sum_{i\ne j}\left(1-\rho_{ij}\right).
\]
Since
\(
1-t\leq-\log t\) for \( 0<t\leq1,\)
we further obtain
\begin{equation}
\label{eq:log-sum-comparison}
\sum_{i\ne j}-\log\rho_{ij}^{(p)}
\leq
\frac{K_p}{\eta_0}
\sum_{i\ne j}-\log\rho_{ij}.
\end{equation}

The infinite product
\(\prod_{i\ne j}\rho_{ij}\) is bounded below by
\(\delta_\Lambda>0\). Hence the series
\(\sum_{i\ne j}-\log\rho_{ij}\) converges, and, by passing to the
limit through finite partial products,
\[
\sum_{i\ne j}-\log\rho_{ij}
=
-\log\left(\prod_{i\ne j}\rho_{ij}\right)
\leq
-\log\delta_\Lambda.
\]
Combining this with \eqref{eq:log-sum-comparison}, we get
\(
\sum_{i\ne j}-\log\rho_{ij}^{(p)}
\leq
-\frac{K_p}{\eta_0}\log\delta_\Lambda.
\)

Exponentiating yields
\(
\prod_{i\ne j}\rho_{ij}^{(p)}
\geq
\delta_\Lambda^{K_p/\eta_0}.
\)
The lower bound is positive and independent of \(j\). Therefore,
\(
\delta_{\Lambda^p}
=
\inf_{j\in\mathbb N}
\prod_{i\ne j}
|b_{\lambda_i^p}(\lambda_j^p)|
\geq
\delta_\Lambda^{K_p/\eta_0}
>0.
\)
Consequently,
\(
\Lambda^p\in\mathcal C.
\)
\end{proof}

\begin{cor}
Let \(\Lambda\subset(0,1)\) and \(p>0\). Then
\(
\Lambda\in\mathcal C\) if and only if
\(
\Lambda^p\in\mathcal C.
\)
\end{cor}

\begin{proof}
If \(\Lambda\in\mathcal C\), then Proposition~\ref{bb} implies that
\(
\Lambda^p\in\mathcal C.
\)

Conversely, suppose that \(\Lambda^p\in\mathcal C\). Since
\(\Lambda^p\subset(0,1)\) and \(1/p>0\), Proposition~\ref{bb}, applied
to \(\Lambda^p\) with exponent \(1/p\), yields
\(
(\Lambda^p)^{1/p}\in\mathcal C.
\)
Since
\(
(\Lambda^p)^{1/p}=\Lambda,
\)
we conclude that \(\Lambda\in\mathcal C\).
\end{proof}

Let
\(
\Lambda=\bigcup_{r=1}^{N}\Lambda_r\subset(0,1),
\
\Lambda_r\in\mathcal C,
\)
and let \(p>0\). The following theorem shows that the points of
\(\Lambda\) and \(\Lambda^p\) can be grouped into the same finite
clusters.

\begin{thm}
\label{thm:simultaneous-cluster-decomposition}
Let
\(
\Lambda=\bigcup_{r=1}^{N}\Lambda_r\subset(0,1),
\
\Lambda_r\in\mathcal C,
\)
and let \(p>0\). Then, for every \(0<\eta<1\), there exists a
partition
\[
\Lambda
=
\mathop{\dot\bigcup}_{n\in\mathbb N^+}\sigma_n,
\qquad
\sigma_n
=
\{\lambda_{n,1},\ldots,\lambda_{n,m_n}\},
\]
such that, upon setting
\(
\sigma_n^p
:=
\{\lambda_{n,1}^p,\ldots,\lambda_{n,m_n}^p\},
\)
one has
\(
\Lambda^p
=
\mathop{\dot\bigcup}_{n\in\mathbb N^+}\sigma_n^p,
\)
and both partitions
\(
\Lambda
=
\mathop{\dot\bigcup}_{n\in\mathbb N^+}\sigma_n
\) and \(
\Lambda^p
=
\mathop{\dot\bigcup}_{n\in\mathbb N^+}\sigma_n^p
\)
satisfy properties \textup{(i)}--\textup{(iv)} of
Lemma~\ref{cc}, with the prescribed parameter \(\eta\) in
property \textup{(ii)}. The constants occurring in
properties \textup{(iii)} and \textup{(iv)} may be different for the
two partitions.
\end{thm}

\begin{proof}
We first record two uniform comparison estimates.
Since each \(\Lambda_r\in\mathcal C\), the set \(\Lambda_r\) has no
accumulation point in \(\mathbb D\). Because
\(\Lambda_r\subset(0,1)\), it follows that
\(
\inf_{\lambda\in\Lambda_r}\lambda>0,
\ 1\leq r\leq N.
\)
Since the union is finite, we have
\(
a
:=
\inf_{\lambda\in\Lambda}\lambda
=
\min_{1\leq r\leq N}
\inf_{\lambda\in\Lambda_r}\lambda
>0.
\)
Therefore, Lemma~\ref{aa} applies to the whole set \(\Lambda\) and
gives constants
\(
0<c_p\leq C_p<\infty
\)
such that
\begin{equation}
\label{eq:uniform-rho-power-comparison}
c_p\rho(\lambda,\mu)
\le
\rho(\lambda^p,\mu^p)
\le
C_p\rho(\lambda,\mu)
\end{equation}
for all distinct \(\lambda,\mu\in\Lambda\).

Moreover, the calculation in
\eqref{eq:one-minus-rho-comparison} gives
\begin{equation}
\label{eq:uniform-one-minus-rho-power}
1-\rho(\lambda^p,\mu^p)
\le
K_p\bigl(1-\rho(\lambda,\mu)\bigr),
\end{equation}
where
\(
K_p
=
2\frac{\max\{1,p\}}{\min\{1,p\}}.
\)

Fix \(0<\eta<1\), and put
\(
\eta_0
:=\frac{\min\{\eta,1/3\}}
     {2\max\{1,C_p\}}.
\)
Then \(0<\eta_0<1\). Applying Lemma~\ref{cc} to \(\Lambda\)
with the parameter \(\eta_0\), we obtain a partition
\(
\Lambda
=
\mathop{\dot\bigcup}_{n\in\mathbb N^+}\sigma_n,
\
\sigma_n
=
\{\lambda_{n,1},\ldots,\lambda_{n,m_n}\},
\)
satisfying properties \textup{(i)}--\textup{(iv)} of
Lemma~\ref{cc}.
Since the mapping \(t\mapsto t^p\) is injective on \((0,1)\), the sets
\(
\sigma_n^p
=
\{\lambda_{n,1}^p,\ldots,\lambda_{n,m_n}^p\}
\)
form a partition
\(
\Lambda^p
=
\mathop{\dot\bigcup}_{n\in\mathbb N^+}\sigma_n^p.
\)
We verify that this partition satisfies the required properties.

\medskip

\noindent
\textup{(i)}
For every \(n\in\mathbb N^+\), injectivity of the power map gives
\(
|\sigma_n^p|
=
|\sigma_n|
=
m_n
\le N.
\)

\medskip

\noindent
\textup{(ii)}
Let \(n\in\mathbb N^+\), and let
\(\lambda,\mu\in\sigma_n\) be distinct. Since the original
decomposition was chosen with parameter \(\eta_0\),
\(
\rho(\lambda,\mu)<\eta_0.
\)
Using \eqref{eq:uniform-rho-power-comparison}, we obtain
\(
\rho(\lambda^p,\mu^p)
\le
C_p\rho(\lambda,\mu)
<
C_p\eta_0
\le
\frac{\eta}{2}
<
\eta.
\)
Thus the powered clusters satisfy property \textup{(ii)}.

Notice also that \(\eta_0<\eta\), so the original clusters
\(\sigma_n\) themselves satisfy property \textup{(ii)} with the
parameter \(\eta\). Moreover, the choice of \(\eta_0\) gives
\(\rho(\lambda,\mu)<\frac13\) and \(\rho(\lambda^p,\mu^p)<\frac13
\) for all \(n\in\mathbb N^+\) and all distinct
\(\lambda,\mu\in\sigma_n\).

\medskip

\noindent
\textup{(iii)}
By property \textup{(iii)} for the original decomposition, there
exists \(\delta>0\) such that, for every choice
\(
\lambda_n^0\in\sigma_n\) for \(n\in\mathbb N^+,\)
the selector
\(
\Lambda_0
=
\{\lambda_n^0:n\in\mathbb N^+\}
\)
belongs to \(\mathcal C\) and satisfies
\(
\delta_{\Lambda_0}\ge\delta.
\)

We prove that the powered selectors satisfy the Carleson condition
with a lower bound that is uniform over all choices. Set
\(
\Lambda_0^p
=
\{(\lambda_n^0)^p:n\in\mathbb N^+\}.
\)
For any two distinct points \(\lambda_i^0,\lambda_j^0\in\Lambda_0\),
we have
\(
\rho(\lambda_i^0,\lambda_j^0)
\ge
\delta_{\Lambda_0}
\ge
\delta.
\)
Indeed, the Carleson product at \(\lambda_j^0\) is no larger than any
of its factors.
It follows from \eqref{eq:uniform-rho-power-comparison} that
\(
\rho\bigl((\lambda_i^0)^p,(\lambda_j^0)^p\bigr)
\ge
c_p\delta.
\)
Put
\(
\varepsilon_p
:=
\min\left\{\frac12,c_p\delta\right\}>0.
\)
Then
\(
\rho\bigl((\lambda_i^0)^p,(\lambda_j^0)^p\bigr)
\ge
\varepsilon_p
\)
for all \(i\ne j\), uniformly over all selectors.

For \(t\in[\varepsilon_p,1]\), one has
\(
-\log t
\le
\frac{1-t}{\varepsilon_p}.
\)
Hence, for every fixed \(j\),
\begin{align*}
\sum_{i\ne j}
-\log
\rho\bigl((\lambda_i^0)^p,(\lambda_j^0)^p\bigr)
&\le
\frac{1}{\varepsilon_p}
\sum_{i\ne j}
\left(
1-
\rho\bigl((\lambda_i^0)^p,(\lambda_j^0)^p\bigr)
\right)
\le
\frac{K_p}{\varepsilon_p}
\sum_{i\ne j}
\left(
1-\rho(\lambda_i^0,\lambda_j^0)
\right)
\\
&\le
\frac{K_p}{\varepsilon_p}
\sum_{i\ne j}
-\log\rho(\lambda_i^0,\lambda_j^0)
\le
-\frac{K_p}{\varepsilon_p}\log\delta.
\end{align*}
In the second inequality we used
\eqref{eq:uniform-one-minus-rho-power}, and in the last inequality we
used \(\delta_{\Lambda_0}\ge\delta\).
Passing first through finite partial products and then to the limit,
we obtain
\(
\prod_{i\ne j}
\rho\bigl((\lambda_i^0)^p,(\lambda_j^0)^p\bigr)
\ge
\delta^{K_p/\varepsilon_p}.
\)
Consequently,
\(
\delta_{\Lambda_0^p}
\ge
\delta_p,
\
\delta_p
:=
\delta^{K_p/\varepsilon_p}>0.
\)
The constant \(\delta_p\) is independent of the choice of the points
\(\lambda_n^0\). Thus property \textup{(iii)} holds for the powered
partition.

\medskip

\noindent
\textup{(iv)}
For each \(n\in\mathbb N^+\), let
\(
B_n^{(p)}(z)
:=
\prod_{\zeta\in\sigma_n^p} b_\zeta(z).
\)
Fix \(n\in\mathbb N^+\) and
\(\zeta\in\sigma_n^p\). For every \(k\ne n\), choose
\(\zeta_k\in\sigma_k^p\) such that
\(
\rho(\zeta,\zeta_k)
=
\min_{\xi\in\sigma_k^p}\rho(\zeta,\xi).
\)
Taking \(\zeta\) as the selected point in the \(n\)-th cluster and
\(\zeta_k\) as the selected point in the \(k\)-th cluster, property
\textup{(iii)} gives
\(
\prod_{k\ne n}\rho(\zeta,\zeta_k)\geq\delta_p.
\)
Since \(|\sigma_k^p|\leq N\), we have
\[
\prod_{k\ne n}|B_k^{(p)}(\zeta)|
=
\prod_{k\ne n}
\prod_{\xi\in\sigma_k^p}\rho(\zeta,\xi)
\geq
\prod_{k\ne n}\rho(\zeta,\zeta_k)^N
\geq
\delta_p^N.
\]
Consequently,
\(
\inf_{n\in\mathbb N^+}
\inf_{\zeta\in\sigma_n^p}
\prod_{k\ne n}
\left|B_k^{(p)}(\zeta)\right|
\geq
\delta_p^N
>0.
\)
The preceding estimate gives the uniform lower bound on the products
of the Blaschke factors corresponding to the other clusters that is
used in Hartmann's construction. Moreover, property \textup{(iii)}
shows that every selector of the powered clusters is a Carleson
sequence with the common lower bound \(\delta_p\), while the final
observation in \textup{(ii)} ensures that distinct points in the same
powered cluster have pseudo-hyperbolic distance less than \(1/3\).
Consequently, the proof of
\cite[Theorem~3.20 and Section~3.5]{AHartmann}
applies to the family
\(
\{\sigma_n^p\}_{n\in\mathbb N^+},
\)
with the preceding estimates replacing the corresponding consequences
of the generalized Carleson condition. Hence there exist a constant
\(M_p<\infty\) and functions
\(
D_n^{(p)}\in H^\infty(\mathbb D),
\ n\in\mathbb N^+,
\)
such that
\[
D_n^{(p)}(\lambda)
=
\begin{cases}
1, & \lambda\in\sigma_n^p,\\
0, & \lambda\in\Lambda^p\setminus\sigma_n^p,
\end{cases}
\]
and
\(
\sum_{n\in\mathbb N^+}
\left|D_n^{(p)}(z)\right|
\leq
M_p,
\
z\in\mathbb D.
\)
Thus property \textup{(iv)} holds for the powered partition.

Therefore the same clustering of the indices gives admissible
decompositions of both \(\Lambda\) and \(\Lambda^p\).
\end{proof}

\section{Trace Spaces under Power Perturbations}
The main result of this section is the trace-space identity
\(
H^2(\Lambda)=H^2(\Lambda^p)
\)
for every finite union of Carleson sequences
\(\Lambda\subset[0,1)\) and every \(p>0\), under the natural coordinate
identification. The proof uses the compatible cluster decompositions
from Section~2 and uniform comparisons of the corresponding divided
differences.

We first recall the ordinary Newton divided differences and their
pseudo-hyperbolic analogue used in Hartmann's trace theorem.

\begin{definition}
Let
\(
z_1,\ldots,z_m\in\mathbb C
\)
be an ordered collection of pairwise distinct points, and let
\(
a(z_j)=a_j\in\mathbb C\) for \(1\le j\le m.\)

The ordinary Newton divided differences are defined by
\(
D^0a[z_j]:=a_j\) for \(1\le j\le m,\)
and, for \(1\le k\le m-1\),
\[
D^ka[z_1,\ldots,z_{k+1}]
:=
\frac{
D^{k-1}a[z_1,\ldots,z_{k-1},z_{k+1}]
-
D^{k-1}a[z_1,\ldots,z_k]
}{
z_{k+1}-z_k
}.
\]
In particular,
\(
D^1a[z_1,z_2]
=
\frac{a_2-a_1}{z_2-z_1}.
\)

Now assume that
\(
\lambda_1,\ldots,\lambda_m\in\mathbb D
\)
are pairwise distinct. The pseudo-hyperbolic divided differences are
defined by
\(
\widetilde D^0a[\lambda_j]:=a_j,\ 1\le j\le m\),
and, for \(1\le k\le m-1\),
\[
\widetilde D^ka[\lambda_1,\ldots,\lambda_{k+1}]
:=
\frac{
\widetilde D^{k-1}a
[\lambda_1,\ldots,\lambda_{k-1},\lambda_{k+1}]
-
\widetilde D^{k-1}a
[\lambda_1,\ldots,\lambda_k]
}{
b_{\lambda_k}(\lambda_{k+1})
}.
\]
In particular,
\(
\widetilde D^1a[\lambda_1,\lambda_2]
=
\frac{a_2-a_1}{b_{\lambda_1}(\lambda_2)}.
\)
\end{definition}

\begin{remark}
The recursion defining \(D^k\) is equivalent to the usual recursive
definition of the classical divided difference. In particular,
\(D^ka[z_1,\ldots,z_{k+1}]\) is symmetric in its nodes.

In contrast,
\(\widetilde D^ka[\lambda_1,\ldots,\lambda_{k+1}]\) generally depends
on the ordering of the nodes. Its definition is adapted to the
pseudo-hyperbolic Newton expansion
\[
a(\lambda_r)
=
\sum_{k=1}^{r}
\widetilde D^{k-1}a[\lambda_1,\ldots,\lambda_k]
\prod_{\ell=1}^{k-1}
b_{\lambda_\ell}(\lambda_r),
\qquad 1\le r\le m,
\]
where an empty product is understood to be equal to \(1\).
\end{remark}

We now introduce the trace spaces and the sequence space appearing in
Hartmann's trace theorem \cite{Hartmann}. Let
\(
X\subset\operatorname{Hol}(\mathbb D)
\)
be a vector space, and let
\(
\Gamma=\{\gamma_j\}_{j\in\mathbb N^+}\subset\mathbb D
\)
be a sequence of distinct points. We write
\(
X|_\Gamma
:=
\{f|_\Gamma:f\in X\}
\subset\mathbb C^\Gamma
\)
for the trace space of \(X\) on \(\Gamma\), and
\(
X(\Gamma)
:=
\left\{
\{f(\gamma_j)\}_{j\in\mathbb N^+}:f\in X
\right\}
\)
for its sequence representation with respect to the prescribed
enumeration of \(\Gamma\).

Let
\(
\Lambda
=
\bigcup_{r=1}^{N}\Lambda_r\) for \(\Lambda_r\in\mathcal C\)
and let
\(
\Lambda
=
\mathop{\dot\bigcup}_{n\in\mathbb N^+}\sigma_n,
\
\sigma_n
=
\{\lambda_{n,1},\ldots,\lambda_{n,m_n}\},
\
m_n:=|\sigma_n|,
\)
be a decomposition satisfying properties
\textup{(i)}--\textup{(iv)} of Lemma~\ref{cc}.
An ordering of the points in each cluster \(\sigma_n\)
is fixed throughout.

Set
\(
\Lambda_0
:=
\{\lambda_{n,1}:n\in\mathbb N^+\}.
\)
By property \textup{(iii)} of Lemma~\ref{cc}, the sequence
\(\Lambda_0\) belongs to \(\mathcal C\). Define
\(
\ell_0(X)
:=
X(\Lambda_0)
=
\left\{
\{g(\lambda_{n,1})\}_{n\in\mathbb N^+}:g\in X
\right\}.
\)

We then define
\[
\ell_{\Lambda,N,\sigma}(X)
:=
\left\{
a\in\mathbb C^\Lambda:
\left\{
\max_{1\le k\le m_n}
\left|
\widetilde D^{\,k-1}a
[\lambda_{n,1},\ldots,\lambda_{n,k}]
\right|
\right\}_{n\in\mathbb N^+}
\in\ell_0(X)
\right\},
\]
where \(\sigma=\{\sigma_n\}_{n\in\mathbb N^+}\).

In particular, when \(X=H^2(\mathbb D)\), the Shapiro--Shields trace
theorem~\cite{ShapiroShields} gives
\[
\ell_0(H^2)
=
\left\{
\{a_n\}_{n\in\mathbb N^+}:
\sum_{n\in\mathbb N^+}
(1-|\lambda_{n,1}|^2)|a_n|^2
<\infty
\right\}.
\]
Equivalently, one may replace \(1-|\lambda_{n,1}|^2\) by
\(1-|\lambda_{n,1}|\), since
\(
1-|\lambda|
\le
1-|\lambda|^2
\le
2(1-|\lambda|),
\ \lambda\in\mathbb D.
\)

\begin{definition}
Let
\(
X\subset\operatorname{Hol}(\mathbb D).
\)
The space \(X\) is called \(\mathcal C\)-stable if, for every pair of
Carleson sequences
\(
\Lambda=\{\lambda_i\}_{i\in\mathbb N^+},
\widetilde{\Lambda}
=
\{\widetilde{\lambda}_i\}_{i\in\mathbb N^+},
\)
satisfying
\(
\sup_{i\in\mathbb N^+}
\left|
b_{\lambda_i}(\widetilde{\lambda}_i)
\right|
<1,
\)
we have
\(
X(\Lambda)=X(\widetilde{\Lambda}).
\)
\end{definition}
Hartmann proved that both \(H^\infty(\mathbb D)\) and
\(H^2(\mathbb D)\) are \(\mathcal C\)-stable; see
\cite[Section~2.1]{Hartmann}.

The following lemma is Hartmann's trace theorem
\cite[Theorem~1.4]{Hartmann}.
\begin{lem}
\label{zz}
Let
\(
X\subset\operatorname{Hol}(\mathbb D)
\)
be a \(\mathcal C\)-stable vector space satisfying
\(
H^\infty(\mathbb D)X\subset X.
\)
Let
\(
\Lambda=\bigcup_{r=1}^{N}\Lambda_r,
\
\Lambda_r\in\mathcal C,
\)
and let
\(
\Lambda
=
\mathop{\dot\bigcup}_{n\in\mathbb N^+}\sigma_n
\)
be any decomposition satisfying properties
\textup{(i)--(iv)} of Lemma~\ref{cc}. Then
\(
X|_\Lambda
=
\ell_{\Lambda,N,\sigma}(X),
\)
where
\(
\sigma=\{\sigma_n\}_{n\in\mathbb N^+}.
\)
\end{lem}

\begin{lem}
\label{lem:h2-finite-trace}
Let
\(
\Lambda=\bigcup_{r=1}^{N}\Lambda_r,
\Lambda_r\in\mathcal C,
\)
and let
\[
\Lambda
=
\mathop{\dot\bigcup}_{n\in\mathbb N^+}\sigma_n,
\qquad
\sigma_n
=
\{\lambda_{n,1},\ldots,\lambda_{n,m_n}\},
\qquad
m_n:=|\sigma_n|,
\]
be a decomposition satisfying properties \textup{(i)--(iv)} of
Lemma~\ref{cc}. Set
\(
\sigma:=\{\sigma_n\}_{n\in\mathbb N^+}.
\)
Then
\(
H^2(\mathbb D)|_\Lambda
=
\ell_{\Lambda,N,\sigma}\bigl(H^2(\mathbb D)\bigr).
\)

Equivalently, a function \(a:\Lambda\to\mathbb C\) belongs to
\(H^2(\mathbb D)|_\Lambda\) if and only if
\begin{equation}
\label{eq:h2-trace-condition-lambda}
\sum_{n\in\mathbb N^+}
\left(1-|\lambda_{n,1}|\right)
\left(
\max_{1\le k\le m_n}
\left|
\widetilde D^{\,k-1}a
[\lambda_{n,1},\ldots,\lambda_{n,k}]
\right|
\right)^2
<\infty.
\end{equation}
\end{lem}

\begin{proof}
Since \(H^2(\mathbb D)\) is \(\mathcal C\)-stable and
\(
H^\infty(\mathbb D)H^2(\mathbb D)
\subset H^2(\mathbb D),
\)
Lemma~\ref{zz}, applied with \(X=H^2(\mathbb D)\), gives
\(
H^2(\mathbb D)|_\Lambda
=
\ell_{\Lambda,N,\sigma}\bigl(H^2(\mathbb D)\bigr).
\)

Let
\(
\Lambda_0
=
\{\lambda_{n,1}:n\in\mathbb N^+\}.
\)
By property \textup{(iii)} of Lemma~\ref{cc}, the sequence
\(\Lambda_0\) belongs to \(\mathcal C\). The Shapiro--Shields trace
theorem~\cite{ShapiroShields} yields
\[
H^2(\mathbb D)(\Lambda_0)
=
\left\{
\{c_n\}_{n\in\mathbb N^+}:
\sum_{n\in\mathbb N^+}
\left(1-|\lambda_{n,1}|\right)|c_n|^2
<\infty
\right\}.
\]
Therefore, by the definition of
\(\ell_{\Lambda,N,\sigma}(H^2(\mathbb D))\), a function
\(a:\Lambda\to\mathbb C\) belongs to \(H^2(\mathbb D)|_\Lambda\) if
and only if \eqref{eq:h2-trace-condition-lambda} holds.
\end{proof}

The following estimate for ordinary divided differences is a standard
consequence of the Genocchi--Hermite formula; see
\cite{deBoorDividedDifferences}.

\begin{lem}
\label{lem:divided-difference-estimate}
Let \(k\in\mathbb N^+\), let \(I\subset\mathbb R\) be an interval,
and let \(z_1,\ldots,z_k\in I\) be pairwise distinct. If
\(F\in C^{k-1}(I;\mathbb C)\), then
\[
\left|
D^{k-1}F[z_1,\ldots,z_k]
\right|
\le
\frac{1}{(k-1)!}
\sup_{z\in J}|F^{(k-1)}(z)|,
\]
where
\(
J
:=
\left[
\min_{1\le j\le k}z_j,
\max_{1\le j\le k}z_j
\right].
\)
\end{lem}
The next proposition compares pseudo-hyperbolic divided differences
with ordinary divided differences after the hyperbolic change of
variables. It will play a key role in the proof of the power stability
result.

\begin{prop}
\label{prop:pseudo-hyperbolic-ordinary-comparison}
Let \(M\in\mathbb N^+\) and \(L>0\). There exists a constant
\(
C=C(M,L)\geq1
\)
such that the following holds.

Let
\(
u_1,\ldots,u_m>0,
\
1\leq m\leq M,
\)
be pairwise distinct points satisfying
\(
|u_i-u_j|\leq L
\)
for \(1\leq i,j\leq m\). Put
\(
\lambda_i=\tanh u_i,
\)
and denote
\(
u^{(k)}=(u_1,\ldots,u_k),
\
\lambda^{(k)}=(\lambda_1,\ldots,\lambda_k).
\)
Then, for every vector
\(
a=(a_1,\ldots,a_m)\in\mathbb C^m,
\)
we have
\begin{equation}
\label{eq:pseudo-ordinary-comparison}
C^{-1}
\max_{1\leq k\leq m}
\left|
D^{k-1}a[u^{(k)}]
\right|
\leq
\max_{1\leq k\leq m}
\left|
\widetilde D^{k-1}a[\lambda^{(k)}]
\right|
\leq
C
\max_{1\leq k\leq m}
\left|
D^{k-1}a[u^{(k)}]
\right|.
\end{equation}
\end{prop}

\begin{proof}
For \(1\leq k\leq m\), define the ordinary Newton basis
\(
N_k(u)
:=
\prod_{\ell=1}^{k-1}(u-u_\ell)
\)
and the pseudo-hyperbolic Newton basis
\(
W_k(u)
:=
\prod_{\ell=1}^{k-1}
b_{\lambda_\ell}(\tanh u),
\)
where an empty product is understood to be equal to \(1\).

Since
\(
\lambda_\ell=\tanh u_\ell>0,
\)
we have
\[
b_{\lambda_\ell}(\tanh u)
=
\frac{\tanh u_\ell-\tanh u}
{1-\tanh u_\ell\tanh u}
=
\tanh(u_\ell-u).
\]
Hence
\(
W_k(u)
=
N_k(u)
\prod_{\ell=1}^{k-1}\gamma_\ell(u),
\)
where
\[
\gamma_\ell(u)
:=
\begin{cases}
\dfrac{\tanh(u_\ell-u)}{u-u_\ell},
& u\neq u_\ell,\\[1.2ex]
-1,
& u=u_\ell.
\end{cases}
\]

Let
\(
I
:=
\left[
\min_{1\leq i\leq m}u_i,
\max_{1\leq i\leq m}u_i
\right].
\)
Since
\(
\operatorname{diam}(I)\leq L,
\)
we have
\(
|u-u_\ell|\leq L,
\
u\in I.
\)
The function
\(
s\mapsto\frac{\tanh s}{s}
\)
extends continuously to a positive smooth function at \(s=0\).
Therefore, there exist constants
\(
0<c_L\leq C_L<\infty,
\)
depending only on \(L\), such that
\begin{equation}
\label{eq:gamma-uniform-bounds}
c_L
\leq
|\gamma_\ell(u)|
\leq
C_L,
\qquad
u\in I.
\end{equation}

For \(1\leq k,r\leq m\), define
\(
\alpha_{k,r}
:=
D^{r-1}W_k[u^{(r)}].
\)
Since
\(
W_k(u_j)=0,
\
1\leq j<k,
\)
we have
\(
\alpha_{k,r}=0,
\
1\leq r<k.
\)
The ordinary Newton interpolation formula therefore gives
\begin{equation}
\label{eq:Wk-expansion}
W_k(u_i)
=
\sum_{r=k}^{i}
\alpha_{k,r}N_r(u_i),
\qquad
1\leq i\leq m,
\end{equation}
where the sum is understood to be zero when \(i<k\).

For the diagonal coefficients, since
\(W_k(u_1)=\cdots=W_k(u_{k-1})=0\), we obtain
\[
\alpha_{k,k}
=
D^{k-1}W_k[u^{(k)}]
=
\frac{W_k(u_k)}{N_k(u_k)}
=
\prod_{\ell=1}^{k-1}\gamma_\ell(u_k).
\]
By \eqref{eq:gamma-uniform-bounds} and \(k\leq m\leq M\), there
exist constants
\(
0<c_0\leq C_0<\infty,
\)
depending only on \(M\) and \(L\), such that
\begin{equation}
\label{eq:alpha-diagonal}
c_0
\leq
|\alpha_{k,k}|
\leq
C_0,
\qquad
1\leq k\leq m.
\end{equation}

Each function
\(
u\mapsto\tanh(u_\ell-u)
\)
and all its derivatives are uniformly bounded on \(\mathbb R\),
independently of \(u_\ell\). Since \(W_k\) is a product of at most
\(M-1\) such functions, the Leibniz rule shows that there exists a
constant \(K_M<\infty\), depending only on \(M\), such that
\[
\sup_{u\in I}
\left|
W_k^{(j)}(u)
\right|
\leq
K_M,
\qquad
0\leq j\leq M-1,
\quad
1\leq k\leq m.
\]
Hence Lemma~\ref{lem:divided-difference-estimate} yields
\begin{equation}
\label{eq:alpha-bound}
|\alpha_{k,r}|
=
\left|
D^{r-1}W_k[u^{(r)}]
\right|
\leq
\frac{1}{(r-1)!}
\sup_{u\in I}
\left|
W_k^{(r-1)}(u)
\right|
\leq
K_M
\end{equation}
for \(1\leq k\leq r\leq m\).
Now define
\(
A_r
:=
D^{r-1}a[u^{(r)}],
\
1\leq r\leq m,
\)
and
\(
B_k
:=
\widetilde D^{k-1}a[\lambda^{(k)}],
\
1\leq k\leq m.
\)
The ordinary Newton interpolation formula gives
\(
a_i
=
\sum_{r=1}^{i}
A_rN_r(u_i),
\
1\leq i\leq m,
\)
whereas the pseudo-hyperbolic Newton interpolation formula gives
\(
a_i
=
\sum_{k=1}^{i}
B_kW_k(u_i),
\
1\leq i\leq m.
\)
Using \eqref{eq:Wk-expansion}, we obtain
\(
a_i
=
\sum_{r=1}^{i}
\left(
\sum_{k=1}^{r}
\alpha_{k,r}B_k
\right)
N_r(u_i).
\)
By the uniqueness of the ordinary Newton coefficients,
\begin{equation}
\label{eq:triangular-relation}
A_r
=
\sum_{k=1}^{r}
\alpha_{k,r}B_k,
\qquad
1\leq r\leq m.
\end{equation}

It follows directly from \eqref{eq:alpha-bound} that
\(
|A_r|
\leq
rK_M
\max_{1\leq k\leq r}|B_k|
\leq
MK_M
\max_{1\leq k\leq m}|B_k|.
\)
Consequently,
\begin{equation}
\label{eq:A-by-B}
\max_{1\leq r\leq m}|A_r|
\leq
MK_M
\max_{1\leq k\leq m}|B_k|.
\end{equation}

Conversely, solving \eqref{eq:triangular-relation} recursively gives
\(
B_r
=
\alpha_{r,r}^{-1}
\left(
A_r-
\sum_{k=1}^{r-1}
\alpha_{k,r}B_k
\right).
\)
Define
\(
\kappa_1:=c_0^{-1}
\)
and, for \(2\leq r\leq M\), define recursively
\(
\kappa_r
:=
\max
\left\{
\kappa_{r-1},
\,
c_0^{-1}
\left(
1+(r-1)K_M\kappa_{r-1}
\right)
\right\}.
\)
We claim that
\begin{equation}
\label{eq:recursive-B-bound}
\max_{1\leq k\leq r}|B_k|
\leq
\kappa_r
\max_{1\leq j\leq r}|A_j|,
\qquad
1\leq r\leq m.
\end{equation}
For \(r=1\), this follows from
\(
A_1=\alpha_{1,1}B_1
\)
and \eqref{eq:alpha-diagonal}. Suppose that
\eqref{eq:recursive-B-bound} holds with \(r-1\) in place of \(r\).
Using \eqref{eq:alpha-diagonal}, \eqref{eq:alpha-bound}, and the
induction hypothesis, we obtain
\[
|B_r|
\leq
c_0^{-1}
\left(
|A_r|
+
K_M
\sum_{k=1}^{r-1}|B_k|
\right)
\leq
c_0^{-1}
\left(
1+(r-1)K_M\kappa_{r-1}
\right)
\max_{1\leq j\leq r}|A_j|
\leq
\kappa_r
\max_{1\leq j\leq r}|A_j|.
\]
Together with the induction hypothesis, this proves
\eqref{eq:recursive-B-bound}.

Since \(m\leq M\), we conclude that
\begin{equation}
\label{eq:B-by-A}
\max_{1\leq k\leq m}|B_k|
\leq
\kappa_M
\max_{1\leq r\leq m}|A_r|.
\end{equation}

Finally, taking
\(
C
:=
\max
\left\{
1,
MK_M,
\kappa_M
\right\},
\)
the estimates \eqref{eq:A-by-B} and \eqref{eq:B-by-A} yield
\eqref{eq:pseudo-ordinary-comparison}.
\end{proof}

We now prove that, for arbitrary data indexed by \(\Lambda\), the
summability condition in Lemma~\ref{lem:h2-finite-trace} is equivalent
to the corresponding summability condition for the same coordinate
data indexed by \(\Lambda^p\).

\begin{thm}
\label{thm:h2-trace-power-stability}
Let
\(
\Lambda=\{\lambda_j\}_{j\in\mathbb N}\subset[0,1)
\)
be a finite union of Carleson sequences, and let \(p>0\). Enumerate
\(
\Lambda^p=\{\lambda_j^p\}_{j\in\mathbb N}
\)
according to the same index set. Then
\(
H^2(\Lambda)=H^2(\Lambda^p).
\)
\end{thm}

\begin{proof}
We first assume that \(0\notin\Lambda\). Since \(\Lambda\) is a
finite union of Carleson sequences contained in \((0,1)\), it is
bounded away from \(0\). Thus there exists \(\lambda_*>0\) such that
\(
\lambda_*\leq\lambda<1\) for \(\lambda\in\Lambda.\)

Fix \(0<\eta<1\). By Theorem~\ref{thm:simultaneous-cluster-decomposition}, there exists a common
finite-cluster decomposition
\(
\Lambda
=
\mathop{\dot\bigcup}_{n\in\mathbb N^+}\sigma_n,
\
\sigma_n
=
\{\lambda_{n,1},\ldots,\lambda_{n,m_n}\}\) for \(m_n\leq N,\)
such that, upon setting
\(
\sigma_n^p
:=
\{\lambda_{n,1}^p,\ldots,\lambda_{n,m_n}^p\},
\)
both decompositions
\(
\Lambda
=
\mathop{\dot\bigcup}_{n\in\mathbb N^+}\sigma_n\) and
\(\Lambda^p
=
\mathop{\dot\bigcup}_{n\in\mathbb N^+}\sigma_n^p
\)
satisfy the hypotheses of Lemma~\ref{lem:h2-finite-trace}. Moreover,
for distinct points in the same cluster,
\(
\rho(\lambda_{n,i},\lambda_{n,j})<\eta\) and
\(\rho(\lambda_{n,i}^p,\lambda_{n,j}^p)<\eta.
\)

For \(n\in\mathbb N^+\) and \(1\leq k\leq m_n\), put
\(
u_{n,k}:=\operatorname{arctanh}\lambda_{n,k},
\
v_{n,k}:=\operatorname{arctanh}(\lambda_{n,k}^p).
\)
Since
\(
\rho(\tanh u,\tanh v)=|\tanh(u-v)|,
\ u,v>0,
\)
we have
\begin{equation}
\label{eq:trace-power-cluster-diameter}
|u_{n,i}-u_{n,j}|\leq L,
\qquad
|v_{n,i}-v_{n,j}|\leq L,
\end{equation}
where
\(
L:=\operatorname{arctanh}\eta.
\)

Define
\(
\Phi_p(u)
:=
\operatorname{arctanh}\bigl((\tanh u)^p\bigr),
\ u>0.
\)
Then \(v_{n,k}=\Phi_p(u_{n,k})\). Set
\(
u_*:=\operatorname{arctanh}\lambda_*>0.
\)
The map \(\Phi_p\) is an increasing smooth diffeomorphism of
\((0,\infty)\) onto itself. Moreover, there exist constants
\(
0<c_\Phi\leq C_\Phi<\infty
\)
such that
\begin{equation}
\label{eq:trace-power-Phi-derivative-bounds}
c_\Phi
\leq
\Phi_p'(u)
\leq
C_\Phi,
\qquad u\geq u_*,
\end{equation}
and all derivatives of \(\Phi_p\) and \(\Phi_p^{-1}\) up to order
\(N\) are uniformly bounded on the corresponding half-lines.

Indeed, if \(x=e^{-2u}\), then
\(
e^{-2\Phi_p(u)}=R_p(x),
\)
where
\(
R_p(x)
:=
\frac{(1+x)^p-(1-x)^p}
     {(1+x)^p+(1-x)^p}.
\)
Put
\(
q_*:=e^{-2u_*}<1\) and
\(r_p(x):=\frac{R_p(x)}{x}\) for \(x>0.\)
The function \(r_p\) extends to a strictly positive
\(C^\infty\)-function on \([0,q_*]\), with \(r_p(0)=p\).
Consequently,
\(
\Phi_p(u)
=
u-\frac12\log r_p(e^{-2u}),
\ u\geq u_*.
\)
Since \(r_p\) is strictly positive on the compact interval
\([0,q_*]\), the functions \(r_p\), \(1/r_p\), and all their
derivatives are bounded there. Repeated application of the chain rule
therefore shows that
\(
\sup_{u\geq u_*}
\left|\Phi_p^{(m)}(u)\right|
<\infty,
\ 1\leq m\leq N.
\)

Furthermore,
\(
\Phi_p'(u)
=
\frac{
p(\tanh u)^{p-1}\operatorname{sech}^2u
}{
1-(\tanh u)^{2p}
}
>0
\)
and
\(
\Phi_p'(u)\to1
\)
as \(u\to\infty\). Continuity and positivity therefore give
\eqref{eq:trace-power-Phi-derivative-bounds}.

Finally,
\(
(\Phi_p^{-1})'(v)
=
\frac{1}{
\Phi_p'\bigl(\Phi_p^{-1}(v)\bigr)
}.
\)
Repeated differentiation expresses every derivative of
\(\Phi_p^{-1}\) of order at most \(N\) as a finite sum of products of
derivatives of \(\Phi_p\), divided by powers of \(\Phi_p'\).
The preceding derivative bounds and the lower bound in
\eqref{eq:trace-power-Phi-derivative-bounds} therefore imply
\[
\sup_{v\geq\Phi_p(u_*)}
\left|(\Phi_p^{-1})^{(m)}(v)\right|
<\infty,
\qquad 1\leq m\leq N.
\]

We next compare the ordinary divided differences before and after the
change of variables \(v=\Phi_p(u)\). We claim that there exists a
constant \(C_0\geq1\), independent of \(n\) and of the data, such
that, for every \(a_1,\ldots,a_{m_n}\in\mathbb C\),
\begin{align}
&C_0^{-1}
\max_{1\leq r\leq m_n}
\left|
D^{r-1}a[u_{n,1},\ldots,u_{n,r}]
\right|
\nonumber\\
&\qquad\leq
\max_{1\leq r\leq m_n}
\left|
D^{r-1}a[v_{n,1},\ldots,v_{n,r}]
\right|
\nonumber\\
&\qquad\leq
C_0
\max_{1\leq r\leq m_n}
\left|
D^{r-1}a[u_{n,1},\ldots,u_{n,r}]
\right|.
\label{eq:trace-power-ordinary-comparison}
\end{align}

To prove the claim, fix a cluster, suppress the index \(n\), and put
\(
m:=m_n,
\
u_k:=u_{n,k},
\
v_k:=v_{n,k}.
\)

Let
\(
A_k:=D^{k-1}a[u_1,\ldots,u_k],
\
B_k:=D^{k-1}a[v_1,\ldots,v_k].
\)
Define the Newton basis polynomials
\(
N_k(u):=\prod_{\ell=1}^{k-1}(u-u_\ell),
\
M_k(v):=\prod_{\ell=1}^{k-1}(v-v_\ell),
\)
and let \(G_p:=\Phi_p^{-1}\). For \(1\leq k\leq m\), set
\(
R_k(v):=N_k(G_p(v)).
\)
Since \(R_k(v_j)=0\) for \(j<k\), its Newton expansion at
\(v_1,\ldots,v_m\) has the form
\(
R_k(v_i)
=
\sum_{r=k}^{i}\beta_{k,r}M_r(v_i),
\
1\leq i\leq m,
\)
where
\(
\beta_{k,r}
:=
D^{r-1}R_k[v_1,\ldots,v_r].
\)
For the diagonal coefficients,
\[
\beta_{k,k}
=
\frac{N_k(u_k)}{M_k(v_k)}
=
\prod_{\ell=1}^{k-1}
\frac{u_k-u_\ell}
     {\Phi_p(u_k)-\Phi_p(u_\ell)}.
\]
By the mean value theorem and
\eqref{eq:trace-power-Phi-derivative-bounds}, there exist constants
\(0<d_0\leq D_0<\infty\), independent of the cluster, such that
\begin{equation}
\label{eq:trace-power-beta-diagonal}
d_0\leq|\beta_{k,k}|\leq D_0,
\qquad 1\leq k\leq m.
\end{equation}

Let
\(
J_v
:=
[\min_{1\leq j\leq m}v_j,
 \max_{1\leq j\leq m}v_j].
\)
For \(v\in J_v\), one has \(G_p(v)\in J_u\), where
\(
J_u
:=
[\min_{1\leq j\leq m}u_j,
 \max_{1\leq j\leq m}u_j].
\)
Thus, by \eqref{eq:trace-power-cluster-diameter},
\(
|G_p(v)-u_\ell|\leq L\) for \(v\in J_v,
\
1\leq\ell\leq m.\)
The uniform derivative bounds for \(G_p\), together with the Leibniz
rule, imply that there exists \(K<\infty\), independent of the
cluster, such that
\(
\sup_{v\in J_v}|R_k^{(r)}(v)|\leq K\) for \(0\leq r\leq N-1,
\
1\leq k\leq m.\)
Lemma~\ref{lem:divided-difference-estimate} therefore gives
\begin{equation}
\label{eq:trace-power-beta-bound}
|\beta_{k,r}|\leq K,
\qquad
1\leq k\leq r\leq m.
\end{equation}

The ordinary Newton interpolation formula yields
\(
a_i
=
\sum_{k=1}^{i}A_kN_k(u_i)
=
\sum_{k=1}^{i}A_kR_k(v_i).
\)
Using the Newton expansions of the functions \(R_k\), we obtain
\(
a_i
=
\sum_{r=1}^{i}
(
\sum_{k=1}^{r}\beta_{k,r}A_k
)
M_r(v_i).
\)
By uniqueness of the Newton coefficients,
\begin{equation}
\label{eq:trace-power-triangular-relation}
B_r
=
\sum_{k=1}^{r}\beta_{k,r}A_k,
\qquad
1\leq r\leq m.
\end{equation}
By \eqref{eq:trace-power-beta-bound},
\(
\max_{1\leq r\leq m}|B_r|
\leq
NK
\max_{1\leq r\leq m}|A_r|.
\)
Conversely, the triangular system
\eqref{eq:trace-power-triangular-relation}, together with
\eqref{eq:trace-power-beta-diagonal} and
\eqref{eq:trace-power-beta-bound}, can be solved recursively. Since
\(m\leq N\), this gives a constant \(K'\), independent of the
cluster, such that
\(
\max_{1\leq r\leq m}|A_r|
\leq
K'
\max_{1\leq r\leq m}|B_r|.
\)
This proves \eqref{eq:trace-power-ordinary-comparison}.

Let
\(
a=\{a_{n,k}:n\in\mathbb N^+,\ 1\leq k\leq m_n\}
\)
be arbitrary data indexed by \(\Lambda\), and regard the same data as
indexed by \(\Lambda^p\) through
\(\lambda_{n,k}\leftrightarrow\lambda_{n,k}^p\). Define
\[
\Delta_n(a)
:=
\max_{1\leq k\leq m_n}
\left|
\widetilde D^{\,k-1}a
[\lambda_{n,1},\ldots,\lambda_{n,k}]
\right|
\]
and
\[
\Delta_n^{(p)}(a)
:=
\max_{1\leq k\leq m_n}
\left|
\widetilde D^{\,k-1}a
[\lambda_{n,1}^p,\ldots,\lambda_{n,k}^p]
\right|.
\]
Applying Proposition~\ref{prop:pseudo-hyperbolic-ordinary-comparison}
to the nodes \(u_{n,k}\) and \(v_{n,k}\), and then using
\eqref{eq:trace-power-ordinary-comparison}, we obtain a constant
\(C_1\geq1\), independent of \(n\) and \(a\), such that
\begin{equation}
\label{eq:trace-power-pseudo-comparison}
C_1^{-1}\Delta_n(a)
\leq
\Delta_n^{(p)}(a)
\leq
C_1\Delta_n(a).
\end{equation}

Moreover, for every \(0\leq t<1\),
\(
\min\{1,p\}
\leq
\frac{1-t^p}{1-t}
\leq
\max\{1,p\}.
\)
Hence
\begin{equation}
\label{eq:trace-power-weight-comparison}
1-\lambda_{n,1}^p
\asymp
1-\lambda_{n,1},
\end{equation}
with constants independent of \(n\).

Combining
\eqref{eq:trace-power-pseudo-comparison} and
\eqref{eq:trace-power-weight-comparison}, we obtain
\(
\sum_{n\in\mathbb N^+}
(1-\lambda_{n,1})\Delta_n(a)^2
<\infty
\)
if and only if
\(
\sum_{n\in\mathbb N^+}
(1-\lambda_{n,1}^p)\Delta_n^{(p)}(a)^2
<\infty.
\)
Lemma~\ref{lem:h2-finite-trace}, applied to the two compatible
cluster decompositions, now yields
\(
a\in H^2(\Lambda)
\Longleftrightarrow
a\in H^2(\Lambda^p).
\)
This proves the assertion when \(0\notin\Lambda\).

Suppose now that \(0\in\Lambda\), and set
\(
\Gamma:=\Lambda\setminus\{0\}.
\)
If \(\Gamma=\varnothing\), the conclusion is immediate. Otherwise,
\(\Gamma\subset(0,1)\) is a finite union of Carleson sequences, and
the first part of the proof gives
\(
H^2(\Gamma)=H^2(\Gamma^p).
\)
Since \(\Gamma\) is a Blaschke sequence, let \(B_\Gamma\) be the
Blaschke product with zero set \(\Gamma\). As \(0\notin\Gamma\),
\(
B_\Gamma(0)\neq0.
\)
Under the natural coordinate decomposition corresponding to
\(\Lambda=\{0\}\cup\Gamma\), we have
\(
H^2(\Lambda)
=
\mathbb C\oplus H^2(\Gamma).
\)
Indeed, if \(c\in H^2(\Gamma)\) and \(a_0\in\mathbb C\), choose
\(f\in H^2(\mathbb D)\) whose trace on \(\Gamma\) is \(c\), and set
\(
g
:=
f+
\frac{a_0-f(0)}{B_\Gamma(0)}B_\Gamma.
\)
Then \(g(0)=a_0\) and \(g|_\Gamma=f|_\Gamma\). The same argument gives
\(
H^2(\Lambda^p)
=
\mathbb C\oplus H^2(\Gamma^p).
\)
Therefore,
\(
H^2(\Lambda)
=
\mathbb C\oplus H^2(\Gamma)
=
\mathbb C\oplus H^2(\Gamma^p)
=
H^2(\Lambda^p).
\)
The proof is complete.
\end{proof}

\section{Power Stability of Finitely Generated Dynamical Frames for Positive Operators}
We first transfer the trace-space identity from Section~3 to the
weighted evaluation operators and prove
\(
\operatorname{Ran}(T_\Lambda)
=
\operatorname{Ran}(T_{\Lambda^p}).
\)
Combined with the finite-generator characterization of dynamical frames,
this yields the power-stability theorem for positive operators.

To state the next characterization theorem, we introduce the weighted
evaluation operator associated with a sequence in the unit disk. Let
\(
\Lambda=\{\lambda_j\}_{j\in\mathbb N}\subset\mathbb D
\)
and set
\(
\varepsilon_j
:=
\sqrt{1-|\lambda_j|^2},
\
j\in\mathbb N.
\)
Define
\[
T_\Lambda:
D(T_\Lambda)\subset H^2(\mathbb D)
\longrightarrow
\ell^2(\mathbb N)
\]
by
\(
T_\Lambda\varphi
:=
\left\{
\varepsilon_j\varphi(\lambda_j)
\right\}_{j\in\mathbb N},
\)
where
\[
D(T_\Lambda)
:=
\left\{
\varphi\in H^2(\mathbb D):
\left\{
\varepsilon_j\varphi(\lambda_j)
\right\}_{j\in\mathbb N}
\in\ell^2(\mathbb N)
\right\}.
\]

For a finite index set \(I\), set
\(
H_I^2
:=
\bigoplus_{i\in I}H^2(\mathbb D)
\)
and define
\[
T_{\Lambda,I}
:=
\bigoplus_{i\in I}T_\Lambda:
D(T_{\Lambda,I})
\subset H_I^2
\longrightarrow
\ell^2(I\times\mathbb N),
\]
where
\(
D(T_{\Lambda,I})
=
\bigoplus_{i\in I}D(T_\Lambda).
\)
Thus, for
\(
\boldsymbol{\varphi}
=
(\varphi_i)_{i\in I}
\in D(T_{\Lambda,I}),
\)
\(
T_{\Lambda,I}\boldsymbol{\varphi}
=
\left\{
\varepsilon_j\varphi_i(\lambda_j)
\right\}_{(i,j)\in I\times\mathbb N}.
\)

If \(E=\{e_\gamma\}_{\gamma\in\Gamma}\) is a Bessel sequence in \(H\),
we denote its analysis operator by
\[
C_E:H\longrightarrow\ell^2(\Gamma),
\qquad
C_Eh
=
\{\langle h,e_\gamma\rangle\}_{\gamma\in\Gamma}.
\]
 For a finite index set \(I\), we consider dynamical frames generated
by the operator orbits of \(\{f_i\}_{i\in I}\), namely, systems of
the form
\(
\{A^nf_i\}_{n\in\mathbb N,\,i\in I}.
\)

The following characterization lemma for such finitely generated
dynamical frames is due to Cabrelli, Molter, Paternostro, and
Philipp \cite[Theorem~2.3]{CCabrelliUMolter}.

\begin{lem}
\label{ii}
Let \(I\) be a finite index set, let \(A\in\mathcal B(H)\) be a
normal operator, and let \(\{f_i\}_{i\in I}\subset H\). Then the
system
\(
\mathcal A
=
\{A^nf_i\}_{n\in\mathbb N,\,i\in I}
\)
is a frame for \(H\) if and only if the following conditions hold:
\begin{enumerate}[\rm (i)]
\item
The operator \(A\) admits the diagonal normal form
\(
A
=
\sum_{j=0}^{\infty}\lambda_jP_j,
\)
where the series converges in the strong operator topology,
\(
\lambda_j\neq\lambda_k
\) for \(j\neq k\),
\(
\sum_{j=0}^{\infty}P_j=I_H,
\)
and
\(
P_jH=\ker(A-\lambda_jI_H).
\)
Moreover,
\(
\operatorname{mult}(A)
=
\sup_{j\in\mathbb N}\dim P_jH
\le |I|.
\)

\item
The sequence of distinct eigenvalues
\(
\Lambda=\{\lambda_j\}_{j\in\mathbb N}
\subset\mathbb D
\)
is the union of at most \(|I|\) Carleson sequences.

\item
There exist constants \(\alpha,\beta>0\) such that, for every
\(j\in\mathbb N\) and every \(h\in P_jH\),
\[
\alpha(1-|\lambda_j|^2)\|h\|^2
\le
\sum_{i\in I}
|\langle h,P_jf_i\rangle|^2
\le
\beta(1-|\lambda_j|^2)\|h\|^2.
\]

\item
\(
\operatorname{Ran}(T_{\Lambda,I})
+
\ker(C_E^*)
=
\ell^2(I\times\mathbb N),
\)
where
\(
E
=
\left\{
(1-|\lambda_j|^2)^{-1/2}P_jf_i
\right\}_{j\in\mathbb N,\,i\in I}.
\)
\end{enumerate}
\end{lem}
The trace-space identity obtained in the preceding section leads to the
following invariance of the ranges of weighted evaluation operators
under power perturbations.
\begin{thm}
\label{thm:range-power-stability-zero}
Let
\(
\Lambda=\{\lambda_j\}_{j\in\mathbb N}\subset[0,1)
\)
be a finite union of Carleson sequences, and let \(p>0\). Then
\(
\operatorname{Ran}(T_\Lambda)
=
\operatorname{Ran}(T_{\Lambda^p}).
\)
Consequently, for every finite index set \(I\),
\(
\operatorname{Ran}(T_{\Lambda,I})
=
\operatorname{Ran}(T_{\Lambda^p,I}).
\)
\end{thm}
\begin{proof}
For \(j\in\mathbb N\), set
\(
d_j:=\sqrt{1-\lambda_j^2},
\
d_{p,j}:=\sqrt{1-\lambda_j^{2p}},
\)
and define
\(
\omega_j
:=
\frac{d_j}{d_{p,j}}
=
\sqrt{
\frac{1-\lambda_j^2}
     {1-\lambda_j^{2p}}
}.
\)
We first show that
\begin{equation}
\label{eq:weight-ratio-hinfty-traces}
\{\omega_j\}_{j\in\mathbb N},
\quad
\{\omega_j^{-1}\}_{j\in\mathbb N}
\in H^\infty(\Lambda).
\end{equation}

Assume first that \(0\notin\Lambda\). Since \(\Lambda\) is a finite
union of Carleson sequences contained in \((0,1)\), it has no
accumulation point in \(\mathbb D\). Consequently,
\(
\inf_{\lambda\in\Lambda}\lambda>0.
\)
Let
\(
\Lambda
=
\mathop{\dot\bigcup}_{n\in\mathbb N^+}\sigma_n,
\
\sigma_n
=
\{\lambda_{n,1},\ldots,\lambda_{n,m_n}\}\) for \(m_n\leq N,\)

be a decomposition satisfying properties \textup{(i)--(iv)} of
Lemma~\ref{cc}. Put
\(
u_{n,k}:=\operatorname{arctanh}\lambda_{n,k}
\)
and define
\(
\Phi_p(u)
:=
\operatorname{arctanh}\bigl((\tanh u)^p\bigr),
\
m_p(u)
:=
\frac{\cosh\Phi_p(u)}{\cosh u}.
\)
Then
\(
\omega_{n,k}
=
m_p(u_{n,k}).
\)

Since \(\Lambda\) is bounded away from \(0\), there exists \(u_*>0\)
such that \(u_{n,k}\geq u_*\). Writing \(x=e^{-2u}\), we have
\(
e^{-2\Phi_p(u)}
=
R_p(x),
\)
where
\(
R_p(x)
:=
\frac{(1+x)^p-(1-x)^p}
     {(1+x)^p+(1-x)^p}.
\)
Moreover,
\(
r_p(x):=\frac{R_p(x)}{x}
\)
extends to a strictly positive smooth function on
\([0,e^{-2u_*}]\), with \(r_p(0)=p\), and
\(
m_p(u)
=
\frac{1+xr_p(x)}{1+x}\,r_p(x)^{-1/2}.
\)
Define
\[
\widehat m_p(x)
:=
\frac{1+xr_p(x)}{1+x}\,r_p(x)^{-1/2},
\qquad
0\leq x\leq e^{-2u_*}.
\]
Then
\(
m_p(u)=\widehat m_p(e^{-2u}).
\)
Since \(r_p\) is strictly positive and smooth on the compact interval
\([0,e^{-2u_*}]\), both \(\widehat m_p\) and
\(1/\widehat m_p\) are smooth there and have bounded derivatives of
every fixed order. Repeated application of the chain rule therefore
gives
\(
\sup_{u\geq u_*}
(
|m_p^{(r)}(u)|
+
|
(\frac{1}{m_p})^{(r)}(u)
|
)
<\infty,
\
0\leq r\leq N-1.
\)

The hyperbolic diameters of the clusters \(\sigma_n\) are uniformly
bounded. Hence Lemma~\ref{lem:divided-difference-estimate} and
Proposition~\ref{prop:pseudo-hyperbolic-ordinary-comparison} imply
that
\[
\sup_{n\in\mathbb N^+}
\max_{1\leq k\leq m_n}
\left|
\widetilde D^{\,k-1}\omega
[\lambda_{n,1},\ldots,\lambda_{n,k}]
\right|
<\infty
\]
and
\[
\sup_{n\in\mathbb N^+}
\max_{1\leq k\leq m_n}
\left|
\widetilde D^{\,k-1}\omega^{-1}
[\lambda_{n,1},\ldots,\lambda_{n,k}]
\right|
<\infty.
\]
Set
\(
\Lambda_0:=\{\lambda_{n,1}:n\in\mathbb N^+\}
\)
and
\(
\sigma:=\{\sigma_n\}_{n\in\mathbb N^+}.
\)
By property \textup{(iii)} of Lemma~\ref{cc}, the sequence
\(\Lambda_0\) belongs to \(\mathcal C\). Hence Carleson's interpolation
theorem~\cite{Carleson1958} gives
\(
\ell_0\bigl(H^\infty\bigr)
=
H^\infty(\Lambda_0)
=
\ell^\infty(\mathbb N^+).
\)
The two preceding uniform estimates therefore imply that
\(
\omega,\ \omega^{-1}
\in
\ell_{\Lambda,N,\sigma}\bigl(H^\infty(\mathbb D)\bigr).
\)
Since \(H^\infty(\mathbb D)\) is \(\mathcal C\)-stable and
\(
H^\infty(\mathbb D)H^\infty(\mathbb D)
\subset H^\infty(\mathbb D),
\)
Lemma~\ref{zz}, applied with \(X=H^\infty(\mathbb D)\), yields
\(
\omega,\ \omega^{-1}
\in H^\infty(\mathbb D)|_\Lambda.
\)
This proves \eqref{eq:weight-ratio-hinfty-traces} when
\(0\notin\Lambda\).

Suppose now that \(0\in\Lambda\), and set
\(
\Gamma:=\Lambda\setminus\{0\}.
\)
If \(\Gamma=\varnothing\), we may simply take \(h=k=1\).
Otherwise, applying the preceding argument to \(\Gamma\), we obtain
functions
\(
h_\Gamma,k_\Gamma\in H^\infty(\mathbb D)
\)
such that
\(
h_\Gamma(\lambda_j)=\omega_j,
\
k_\Gamma(\lambda_j)=\omega_j^{-1},
\
\lambda_j\in\Gamma.
\)
Let \(B_\Gamma\) be the Blaschke product whose zero set is \(\Gamma\).
Since \(0\notin\Gamma\), we have \(B_\Gamma(0)\neq0\). Define
\(
h(z)
:=
h_\Gamma(z)
+
\frac{1-h_\Gamma(0)}{B_\Gamma(0)}B_\Gamma(z)
\)
and
\(
k(z)
:=
k_\Gamma(z)
+
\frac{1-k_\Gamma(0)}{B_\Gamma(0)}B_\Gamma(z).
\)
Then \(h,k\in H^\infty(\mathbb D)\), and
\(
h(\lambda_j)=\omega_j,
\
k(\lambda_j)=\omega_j^{-1},
\
\lambda_j\in\Gamma.
\)
Moreover,
\(
h(0)=k(0)=1.
\)
Since
\(
\omega(0)=\omega(0)^{-1}=1,
\)
the functions \(h\) and \(k\) interpolate the required data on the
whole sequence \(\Lambda\). Thus, in all cases, there exist
\(h,k\in H^\infty(\mathbb D)\) such that
\(
h(\lambda_j)=\omega_j,
\
k(\lambda_j)=\omega_j^{-1},
\
j\in\mathbb N.
\)

We now prove the range equality. Let
\(
c=\{c_j\}_{j\in\mathbb N}
\in\operatorname{Ran}(T_\Lambda).
\)
Then there exists \(f\in H^2(\mathbb D)\) such that
\(
c_j=d_jf(\lambda_j),
\
j\in\mathbb N.
\)
Since \(h\in H^\infty(\mathbb D)\), the sequence
\(
b_j
:=
\omega_jf(\lambda_j)
=
(hf)(\lambda_j)
\)
belongs to \(H^2(\Lambda)\). By
Theorem~\ref{thm:h2-trace-power-stability},
\(
H^2(\Lambda)=H^2(\Lambda^p).
\)
Hence there exists \(g\in H^2(\mathbb D)\) such that
\(
g(\lambda_j^p)=b_j,
\
j\in\mathbb N.
\)
Therefore,
\(
d_{p,j}g(\lambda_j^p)
=
d_{p,j}\omega_jf(\lambda_j)
=
d_jf(\lambda_j)
=
c_j.
\)
Thus
\(
c\in\operatorname{Ran}(T_{\Lambda^p}),
\)
and consequently
\(
\operatorname{Ran}(T_\Lambda)
\subset
\operatorname{Ran}(T_{\Lambda^p}).
\)

Conversely, let
\(
c\in\operatorname{Ran}(T_{\Lambda^p}).
\)
Then there exists \(g\in H^2(\mathbb D)\) such that
\(
c_j=d_{p,j}g(\lambda_j^p)\) for \(j\in\mathbb N.\)
By Theorem~\ref{thm:h2-trace-power-stability}, the sequence
\(
\{g(\lambda_j^p)\}_{j\in\mathbb N}
\)
belongs to \(H^2(\Lambda)\). Since \(k\in H^\infty(\mathbb D)\),
the sequence
\(
\left\{
\omega_j^{-1}g(\lambda_j^p)
\right\}_{j\in\mathbb N}
\)
also belongs to \(H^2(\Lambda)\). Hence there exists
\(f\in H^2(\mathbb D)\) such that
\(
f(\lambda_j)
=
\omega_j^{-1}g(\lambda_j^p)\) for \(j\in\mathbb N.\)
It follows that
\(
d_jf(\lambda_j)
=
d_j\omega_j^{-1}g(\lambda_j^p)
=
d_{p,j}g(\lambda_j^p)
=
c_j.
\)
Therefore,
\(
c\in\operatorname{Ran}(T_\Lambda).
\)
We conclude that
\(
\operatorname{Ran}(T_\Lambda)
=
\operatorname{Ran}(T_{\Lambda^p}).
\)

Finally, since \(I\) is finite,
\(
\operatorname{Ran}(T_{\Lambda,I})
=
\bigoplus_{i\in I}\operatorname{Ran}(T_\Lambda)
\)
and
\(
\operatorname{Ran}(T_{\Lambda^p,I})
=
\bigoplus_{i\in I}\operatorname{Ran}(T_{\Lambda^p}).
\)
Hence
\(
\operatorname{Ran}(T_{\Lambda,I})
=
\operatorname{Ran}(T_{\Lambda^p,I}),
\)
which completes the proof.
\end{proof}

The preceding range identity yields the main theorem.

\begin{thm}
\label{thm:power-stability-finitely-generated-frames}
Let \(A\in\mathcal B(H)\) be a positive operator,
let \(I\) be a finite index set, and let \(\{f_i\}_{i\in I}\subset H\).
Then, for every \(p>0\), the system
\(
\{A^nf_i\}_{n\in\mathbb N,\,i\in I}
\)
is a frame for \(H\) if and only if
\(
\{A^{pn}f_i\}_{n\in\mathbb N,\,i\in I}
\)
is a frame for \(H\).
\end{thm}

\begin{proof}
Suppose first that
\(
\{A^nf_i\}_{n\in\mathbb N,\,i\in I}
\)
is a frame for \(H\). By Lemma~\ref{ii}\textup{(i)}, the operator
\(A\) necessarily admits a diagonal spectral representation
\(
A=\sum_{j\in\mathbb N}\lambda_jP_j,
\)
where the series converges in the strong operator topology, the
numbers \(\lambda_j\) are pairwise distinct,
\(
\sum_{j\in\mathbb N}P_j=I_H,
\
P_jH=\ker(A-\lambda_jI_H),
\)
and
\(
\operatorname{mult}(A)
=
\sup_{j\in\mathbb N}\dim P_jH
\leq |I|.
\)
Since \(A\) is positive and \(\lambda_j\in\mathbb D\), we have
\(
0\leq\lambda_j<1,
\
j\in\mathbb N.
\)
Thus
\(
\Lambda:=\{\lambda_j:j\in\mathbb N\}\subset[0,1).
\)

By the spectral theorem,
\(
A^p=\sum_{j\in\mathbb N}\lambda_j^pP_j.
\)
Since the map \(t\mapsto t^p\) is strictly increasing on
\([0,1)\), the numbers \(\lambda_j^p\) are pairwise distinct.
Hence \(A^p\) has the same spectral projections as \(A\), and
\(
\operatorname{mult}(A^p)
=
\operatorname{mult}(A)
\leq |I|.
\)
Therefore condition \textup{(i)} of Lemma~\ref{ii} holds for
\(A^p\).

By Lemma~\ref{ii}\textup{(ii)}, the sequence \(\Lambda\) is the
union of at most \(|I|\) Carleson sequences. Write
\(
\Lambda=\bigcup_{r=1}^{q}\Lambda_r,
\
q\leq |I|,
\
\Lambda_r\in\mathcal C.
\)
For each \(1\leq r\leq q\), set
\(
\Gamma_r:=\Lambda_r\setminus\{0\}.
\)
If \(\Gamma_r\neq\varnothing\), then
\(\Gamma_r\subset(0,1)\) and \(\Gamma_r\in\mathcal C\). Hence
Proposition~\ref{bb} gives
\(
\Gamma_r^p\in\mathcal C.
\)
Since adding or removing finitely many points preserves the
Carleson condition, it follows that
\(
\Lambda_r^p
:=
\{\lambda^p:\lambda\in\Lambda_r\}
\in\mathcal C.
\)
The same conclusion is immediate when
\(\Gamma_r=\varnothing\). Consequently,
\(
\Lambda^p
:=
\{\lambda_j^p:j\in\mathbb N\}
=
\bigcup_{r=1}^{q}\Lambda_r^p
\)
is the union of at most \(|I|\) Carleson sequences. Thus condition
\textup{(ii)} of Lemma~\ref{ii} holds for \(A^p\).

By Lemma~\ref{ii}\textup{(iii)}, there exist constants
\(\alpha,\beta>0\) such that, for every \(j\in\mathbb N\) and every
\(h\in P_jH\),
\(
\alpha(1-\lambda_j^2)\|h\|^2
\leq
\sum_{i\in I}|\langle h,P_jf_i\rangle|^2
\leq
\beta(1-\lambda_j^2)\|h\|^2.
\)
For every \(0\leq t<1\), one has
\(
\min\{1,p\}
\leq
\frac{1-t^p}{1-t}
\leq
\max\{1,p\}.
\)
Applying this inequality to \(t=\lambda_j^2\), we obtain
\[
\frac{1}{\max\{1,p\}}
(1-\lambda_j^{2p})
\leq
1-\lambda_j^2
\leq
\frac{1}{\min\{1,p\}}
(1-\lambda_j^{2p}).
\]
Therefore,
\[
\frac{\alpha}{\max\{1,p\}}
(1-\lambda_j^{2p})\|h\|^2
\leq
\sum_{i\in I}|\langle h,P_jf_i\rangle|^2
\leq
\frac{\beta}{\min\{1,p\}}
(1-\lambda_j^{2p})\|h\|^2.
\]
Hence condition \textup{(iii)} of Lemma~\ref{ii} holds for \(A^p\).

It remains to verify condition \textup{(iv)}. Define
\(
E
:=
\left\{
(1-\lambda_j^2)^{-1/2}P_jf_i
\right\}_{j\in\mathbb N,\,i\in I}
\)
and
\(
E_p
:=
\left\{
(1-\lambda_j^{2p})^{-1/2}P_jf_i
\right\}_{j\in\mathbb N,\,i\in I}.
\)
The upper estimates in condition \textup{(iii)} show that both
\(E\) and \(E_p\) are Bessel sequences. We claim that
\(
\ker C_E^*=\ker C_{E_p}^*.
\)
Indeed, let
\(
a=\{a_{i,j}\}_{(i,j)\in I\times\mathbb N}
\in\ell^2(I\times\mathbb N).
\)
Since the subspaces \(P_jH\) are mutually orthogonal, for every
\(j\in\mathbb N\),
\(
P_jC_E^*a
=
\frac{1}{\sqrt{1-\lambda_j^2}}
\sum_{i\in I}a_{i,j}P_jf_i.
\)
Therefore,
\[
a\in\ker C_E^*
\quad\Longleftrightarrow\quad
\sum_{i\in I}a_{i,j}P_jf_i=0
\quad\text{for every }j\in\mathbb N.
\]
Similarly,
\(
P_jC_{E_p}^*a
=
\frac{1}{\sqrt{1-\lambda_j^{2p}}}
\sum_{i\in I}a_{i,j}P_jf_i,
\)
and hence
\[
a\in\ker C_{E_p}^*
\quad\Longleftrightarrow\quad
\sum_{i\in I}a_{i,j}P_jf_i=0
\quad\text{for every }j\in\mathbb N.
\]
Thus
\(
\ker C_E^*=\ker C_{E_p}^*.
\)
By Theorem~\ref{thm:range-power-stability-zero},
\(
\operatorname{Ran}(T_{\Lambda,I})
=
\operatorname{Ran}(T_{\Lambda^p,I}).
\)
Moreover, Lemma~\ref{ii}\textup{(iv)} gives
\(
\operatorname{Ran}(T_{\Lambda,I})
+
\ker C_E^*
=
\ell^2(I\times\mathbb N).
\)
Consequently,
\[
\operatorname{Ran}(T_{\Lambda^p,I})
+
\ker C_{E_p}^*=
\operatorname{Ran}(T_{\Lambda,I})
+
\ker C_E^*
=
\ell^2(I\times\mathbb N).
\]
Thus condition \textup{(iv)} of Lemma~\ref{ii} holds for \(A^p\).
All four conditions in Lemma~\ref{ii} are therefore satisfied by
\(A^p\) and \(\{f_i\}_{i\in I}\). Hence
\(
\{(A^p)^nf_i\}_{n\in\mathbb N,\,i\in I}
\)
is a frame for \(H\). Since
\(
(A^p)^n=A^{pn}\) for \(n\in\mathbb N,\)
we conclude that
\(
\{A^{pn}f_i\}_{n\in\mathbb N,\,i\in I}
\)
is a frame for \(H\).

Conversely, suppose that
\(
\{A^{pn}f_i\}_{n\in\mathbb N,\,i\in I}
\)
is a frame for \(H\). Set
\(
B:=A^p.
\)
Then \(B\) is positive and
\(
\{B^nf_i\}_{n\in\mathbb N,\,i\in I}
=
\{A^{pn}f_i\}_{n\in\mathbb N,\,i\in I}
\)
is a frame for \(H\). Applying the implication proved above to
\(B\) with exponent \(1/p\), we obtain that
\(
\{B^{n/p}f_i\}_{n\in\mathbb N,\,i\in I}
\)
is a frame for \(H\). Since
\(
B^{1/p}=(A^p)^{1/p}=A,
\)
we have
\(
B^{n/p}=(B^{1/p})^n=A^n.
\)
Therefore,
\(
\{A^nf_i\}_{n\in\mathbb N,\,i\in I}
\)
is a frame for \(H\), completing the proof.
\end{proof}

As an immediate consequence of Theorem~\ref{thm:power-stability-finitely-generated-frames}, we obtain a
multi-orbit analogue of the subsampling and redundancy phenomena
observed in \cite{OChristensenMHasannasab,KrishtalMiller}.

\begin{cor}
\label{cor:multi-orbit-infinite-excess}
Let \(A\in\mathcal B(H)\) be a positive operator, let \(I\) be a
finite index set, and let \(\{f_i\}_{i\in I}\subset H\). Suppose that
\(
\mathcal A
:=
\{A^nf_i\}_{n\in\mathbb N,\,i\in I}
\)
is a frame for \(H\). Then, for every \(q\in\mathbb N^+\), the
subfamily
\(
\mathcal A_q
:=
\{A^{qn}f_i\}_{n\in\mathbb N,\,i\in I}
\)
is also a frame for \(H\). In particular, \(\mathcal A\) has infinite
excess.

In addition, the following elementary observation holds: if \(A\) is
invertible, then the frame property is preserved after the removal of
any finite collection of elements of \(\mathcal A\).
\end{cor}

\begin{proof}
The first assertion follows immediately from Theorem~
\ref{thm:power-stability-finitely-generated-frames} by taking \(p=q\).

For \(q\geq2\), the complement of \(q\mathbb N\) in \(\mathbb N\) is
infinite, while the remaining subfamily \(\mathcal A_q\) is a frame.
Thus an infinite collection of elements can be removed from
\(\mathcal A\) without destroying the frame property, and hence
\(\mathcal A\) has infinite excess.

Finally, assume that \(A\) is invertible, and let
\(J\subset\mathbb N\times I\) be finite. If \(J=\varnothing\), there
is nothing to prove. Otherwise, choose \(K\in\mathbb N\) such that
\(
K>\max\{n:(n,i)\in J\}.
\)
Since \(A^K\) is boundedly invertible, the tail
\(
\{A^{n+K}f_i\}_{n\in\mathbb N,\,i\in I}
=
\{A^K(A^nf_i)\}_{n\in\mathbb N,\,i\in I}
\)
is a frame for \(H\). This tail is contained in
\(
\{A^nf_i:(n,i)\in(\mathbb N\times I)\setminus J\}.
\)
The latter system is Bessel because it is a subfamily of
\(\mathcal A\), and it contains a frame for \(H\). Therefore it is
itself a frame for \(H\).
\end{proof}

The following example shows that the positivity assumption in
Theorem~\ref{thm:power-stability-finitely-generated-frames} cannot be
omitted, even for a self-adjoint contraction with simple spectrum
and a single generator.

\begin{example}
\label{ex:failure-with-negative-eigenvalues}
Let
\(
r_n:=1-2^{-n},
\ n\in\mathbb N^+,
\)
and set
\(
\Lambda
:=
\{r_n:n\in\mathbb N^+\}
\cup
\{-r_n:n\in\mathbb N^+\}.
\)
Let
\(
H:=\ell^2(\mathbb N^+)
\)
with canonical orthonormal basis
\(
\{e_j:j\in\mathbb N^+\}.
\)
Define a self-adjoint contraction \(A\in\mathcal B(H)\) by
\(
Ae_{2n}=r_ne_{2n},
\
Ae_{2n-1}=-r_ne_{2n-1},
\ n\in\mathbb N^+,
\)
and define
\(
f
:=
\sum_{n\in\mathbb N^+}
\sqrt{1-r_n^2}\,
\bigl(e_{2n}+e_{2n-1}\bigr).
\)
Then
\(
\{A^kf\}_{k\in\mathbb N}
\)
is a frame for \(H\), whereas
\(
\{A^{2k}f\}_{k\in\mathbb N}
\)
is not a frame for \(H\).
\end{example}

\begin{proof}
We first show that \(\Lambda\) is a Carleson sequence. Put
\(
q_\ell
:=
\frac{1-2^{-\ell}}{1+2^{-\ell}}\) for \(\ell\in\mathbb N^+.\)
For \(m=n+\ell\), a direct computation gives
\(
\rho(r_n,r_{n+\ell})
=
\frac{1-2^{-\ell}}
{1+2^{-\ell}-2^{-(n+\ell)}}
\geq
q_\ell.
\)
The same estimate holds for
\(\rho(r_n,r_{n-\ell})\) whenever \(n-\ell\geq1\). Since
\(
\sum_{\ell\in\mathbb N^+}(1-q_\ell)<\infty,
\)
the product
\(
Q:=\prod_{\ell\in\mathbb N^+}q_\ell
\)
is strictly positive. Hence
\(
\prod_{\substack{m\in\mathbb N^+\\m\neq n}}
\rho(r_n,r_m)
\geq
Q^2\) for \(n\in\mathbb N^+.\)

For the interaction between the positive and negative branches, we
have
\(
\rho(r_n,-r_m)
=
\frac{r_n+r_m}{1+r_nr_m}
\geq
\frac45
\)
and
\(
1-\rho(r_n,-r_m)
=
\frac{(1-r_n)(1-r_m)}{1+r_nr_m}
\leq
2^{-n-m}.
\)
Therefore,
\[
-\log\rho(r_n,-r_m)
\leq
\frac{1-\rho(r_n,-r_m)}
{\rho(r_n,-r_m)}
\leq
\frac54\,2^{-n-m}.
\]
It follows that
\[
\prod_{m\in\mathbb N^+}
\rho(r_n,-r_m)
\geq
\exp\left(
-\frac54
\sum_{m\in\mathbb N^+}2^{-n-m}
\right)
\geq
e^{-5/8}.
\]
Combining these estimates and using symmetry, we obtain
\(
\inf_{\lambda\in\Lambda}
\prod_{\mu\in\Lambda\setminus\{\lambda\}}
\rho(\lambda,\mu)
\geq
Q^2e^{-5/8}
>0.
\)
Thus
\(
\Lambda\in\mathcal C.
\)
Moreover,
\(
\|f\|^2
=
2\sum_{n\in\mathbb N^+}(1-r_n^2)
\leq
4\sum_{n\in\mathbb N^+}2^{-n}
<\infty,
\)
so \(f\in H\).

We now apply Lemma~\ref{ii} with a single generator.
Define
\(
\lambda_{2n}:=r_n,
\
\lambda_{2n-1}:=-r_n,
\
n\in\mathbb N^+,
\)
and let \(P_j\) denote the orthogonal projection onto
\(\operatorname{span}\{e_j\}\). Then
\(
A
=
\sum_{j\in\mathbb N^+}\lambda_jP_j,
\)
all eigenspaces of \(A\) are one-dimensional, and
\(
\{\lambda_j\}_{j\in\mathbb N^+}
=
\Lambda
\in\mathcal C.
\)
For every \(j\in\mathbb N^+\) and every \(h\in P_jH\), we have
\(
|\langle h,P_jf\rangle|^2
=
(1-|\lambda_j|^2)\|h\|^2.
\)
Thus condition \textup{(iii)} of Lemma~\ref{ii} holds with
\(
\alpha=\beta=1.
\)

Furthermore, the normalized spectral family appearing in
Lemma~\ref{ii} is
\[
E
=
\left\{
(1-|\lambda_j|^2)^{-1/2}P_jf
:
j\in\mathbb N^+
\right\}
=
\{e_j:j\in\mathbb N^+\},
\]
which is an orthonormal basis of \(H\). Hence
\(
\ker C_E^*=\{0\}.
\)
Since \(\Lambda\in\mathcal C\), the Shapiro--Shields interpolation
theorem~\cite{ShapiroShields} gives
\(
\operatorname{Ran}(T_\Lambda)
=
\ell^2(\mathbb N^+).
\)
Therefore condition \textup{(iv)} of Lemma~\ref{ii} also holds. It
follows from Lemma~\ref{ii} that
\(
\{A^kf\}_{k\in\mathbb N}
\)
is a frame for \(H\).

On the other hand, for each \(n\in\mathbb N^+\), the nonzero vector
\(
g_n:=e_{2n}-e_{2n-1}
\)
is orthogonal to every even iterate of \(f\). Indeed, for every
\(k\in\mathbb N\),
\[
\langle g_n,A^{2k}f\rangle=
\sqrt{1-r_n^2}\,r_n^{2k}
-
\sqrt{1-r_n^2}\,r_n^{2k}=
0.
\]
Thus
\(
\{A^{2k}f\}_{k\in\mathbb N}
\)
is not complete in \(H\), and hence it is not a frame. This shows
that the positivity assumption in
Theorem~\ref{thm:power-stability-finitely-generated-frames} cannot be
omitted.
\end{proof}

\section*{Declarations}

\noindent\textbf{Data availability.}
No data were used in this study.

\medskip

\noindent\textbf{Conflict of interest.}
The author declares that there is no conflict of interest.

\medskip

\noindent\textbf{Funding.}
The author received no funding for this work.

\end{document}